\documentclass[11pt]{article}
\usepackage[top=2cm, bottom=2cm, left=2.5cm, right=2.5cm] {geometry}

\usepackage{comment}
\usepackage{amsmath,amssymb,theorem,color}
\numberwithin{equation}{section} 
\usepackage{amssymb}
\usepackage{color}
\usepackage{amsmath}
\usepackage{graphicx}
\usepackage{amsfonts}%
\newcommand{\R}{\ensuremath{\mathbb{R}}}

\newcommand{\N}{\ensuremath{\mathbb{N}}}

\newcommand{\K}{\mathbb{K}}

\newcommand{\cC}{\mathcal{C}}

\newcommand{\X}{\mathbb{X}}

\newcommand{\cD}{\mathcal{D}}
\newcommand{\cL}{\mathcal{L}}
\newcommand{\cM}{\mathcal{M}}
\newcommand{\bx}{\mathbf x}
\newcommand{\tp}{p,\sigma}
\newcommand{\tq}{q,\sigma}

\newcommand{\ltn}{\ensuremath{\left| \! \left| \! \left|}}
\newcommand{\rtn}{\ensuremath{\right| \! \right| \! \right|}}

\newtheorem{theorem}{Theorem}[section]
{ \theorembodyfont{\normalfont} 

}

\newtheorem{lemma}[theorem]{Lemma}
\newtheorem{corollary}[theorem]{Corollary}
\newtheorem{proposition}[theorem]{Proposition}

\newcounter{enumctr}

\begin{document}

\title{Stability theory for Gaussian rough differential equations. Part II.}

\author{Luu Hoang Duc\\Max-Planck-Institut f\"ur Mathematik in den Naturwissenschaften, $\&$\\Institute of Mathematics, Vietnam Academy of Science and Technology\\ {\it E-mail: duc.luu@mis.mpg.de, lhduc@math.ac.vn}
}
\date{}
\maketitle

\begin{abstract}
We propose a quantitative direct method of proving the stability result for Gaussian rough differential equations in the sense of Gubinelli \cite{gubinelli}. Under the strongly dissipative assumption of the drift coefficient function, we prove that the trivial solution of the system under small noise is exponentially stable. 
\end{abstract}



\section{Introduction}

The paper continues our study in the first part \cite{duc19part1} to deal with the asymptotic stability criteria for rough differential equations of the form
\begin{equation}\label{RDE1}
dy_t = [A y_t + f(y_t)] dt + g(y_t)dx_t,
\end{equation}
or in the integral form
\begin{equation}\label{RDE2}
y_t = y_a + \int_a^t [A y_u + f(y_u) ]du + \int_a^t g(y_u) dx_u,\qquad t\in [a,T];
\end{equation}
where the nonlinear part $f: \R^d \to \R^d$ is globally Lipschitz function for simplicity and $g = (g_1,\ldots, g_m)$ is a collection of vector fields $g_j: \R^d \to \R^d$ such that $g_j \in C^{3}_b(\R^d,\R^d)$. Equation \eqref{RDE1} can be viewed as a controlled differential equation driven by rough path $x \in C^{\nu}([a,T],\R^m)$ for $\nu \in (\frac{1}{3},\frac{1}{2})$, in the sense of Lyons \cite{lyons98}, \cite{lyonsetal07} where $x$ can also be considered as an element of the space $C^{p-{\rm var}}([a,T],\R^m)$ of finite $p$ - variation norm, with $p\nu \geq 1$. For instance, given $\bar{\nu} \in (\frac{1}{3},1]$, the path $x$ might be a realization of a $\R^m$-valued centered Gaussian process satisfying: there exists for any $T>0$ a constant $C_T$ such that for all $p \geq \frac{1}{\bar{\nu}}$
\begin{equation}\label{Gaussianexpect}
E \|X_t- X_s\|^{p} \leq C_T |t-s|^{p\bar{\nu}},\quad \forall s,t \in [0,T]. 
\end{equation}
By Kolmogorov theorem, for any $\nu \in (0,\bar{\nu})$ and any interval $[0,T]$ almost all realization of $X$ will be in $C^\nu([0,T])$. Such a stochastic process, in particular, can be a fractional Brownian motion $B^H$  \cite{mandelbrot} with Hurst exponent $H \in (\frac{1}{3},\frac{1}{2})$, i.e. a family of $B^H = \{B^H_t\}_{t\in \R}$ with continuous sample paths and 
\[
E \|B^H_t- B^H_s\| = |t-s|^{2H}, \forall t,s \in \R.
\]
In this paper, we would like to approach system \eqref{RDE1}, where the second integral is well-understood as rough integral in the sense of Gubinelli \cite{gubinelli}. 
Such system satisfies the existence and uniqueness of solution given initial conditions, see e.g. \cite{gubinelli} or \cite{frizhairer} for a version without drift coefficient function, and \cite{riedelScheutzow} for a full version using $p$ - variation norms.\\
To study the local stability, we impose conditions for matrices $A\in \R^{d\times d}$ such that $A$ is negative definite, i.e. there exists a $\lambda >0$ such that
\begin{equation}\label{lambda}
\langle y, Ay \rangle \leq - \lambda_A \|y\|^2.
\end{equation}
We also assume that the nonlinear part $f: \R^d \to \R^d$ is locally Lipschitz function such that 
\begin{equation}\label{condf}
f(0) =0\quad \text{and} \quad \|f(y)\| \leq \|y\| h(\|y\|)
\end{equation}
where $h: \R^+ \to \R^+$ is an increasing function which is bounded above by a constant $C_f$. Our assumption is somehow still global, but it has an advantage of being able to treat the local dynamics as well. We refer to \cite{garrido-atienzaetal} and \cite{GABSch18} for real local versions on a small neighborhood $B(0,\rho)$ of the trivial solution, using the cutoff technique.\\  
In this paper, we also assume that $g(0) = 0$ and $g \in C^{3}_b$ in case $\nu\in (\frac{1}{3},\frac{1}{2})$ with bounded derivatives $C_g$ (which also include the Lipschit coefficient of the highest derivative). System \eqref{RDE1} then admits an equilibrium which is the trivial solution. Our main stability results are then formulated as follows.

\begin{theorem}[Stability for rough systems]\label{stablinRDE}
	Assume $X_\cdot(\omega)$ is a centered Gaussian process with stationary increments satisfying \eqref{Gaussianexpect}, and $\frac{1}{2}>\bar{\nu}>\nu >\frac{1}{3}$ is fixed. Assume further that conditions \eqref{lambda}, \eqref{condf} are satisfied, where $\lambda_A > h(0)$.Then there exists an $\epsilon >0$ such that given $C_g < \epsilon$, and for almost sure all realizations $x_\cdot=X_\cdot(\omega)$, the zero solution of \eqref{RDE1} is locally exponentially stable. If in addition $\lambda_A > C_f$, then we can choose $\epsilon$ so that the zero solution of \eqref{RDE1} is globally exponentially stable a.s. 
\end{theorem}
Our method motivates from the direct method of Lyapunov, which aims to estimate the norm growth (or a Lyapunov-type function) of the solution in discrete intervals using the rough estimates for the angular equation which is feasible thanks to the change of variable formula for rough integral defined in the sense of Gubinelli. It is then sufficient to study the local and global exponential stablity of the corresponding random differential inequality, which can be done with random norm techniques in \cite{arnold}. A necessary assumption is the integrability of solution, which is straightforward for Young equations but  difficult for the rough case under the H\"older norm. Fortunately, we are able to build a modified version of {\it greedy times} in \cite{cassetal} for elements in the $C^{\tp}$ space, which is a little more regular than $C^{p {\rm -var}}$ by respecting also a small H\"older regularity $\sigma$. In addition, under the stronger assumption that the {\it rectangular increments} of the covariance defined by
\[
R \Big(\begin{array}{cc} s&t \\ s^\prime & t^\prime \end{array}\Big)  := E(X_{s,t} \otimes X_{s^\prime,t^\prime}) 
\]
is of finite $(\tq)$ - variation, we prove a similar result to \cite[Theorem 6.3]{cassetal} on the main tail estimate of the number of greedy time under the new $(\tp)$ - norm. The integrability of the solution under the new $(\tp)$ - variation seminorm is then proved in Theorem \ref{integrabilitysolution}. \\
We close the introduction part with a note that our method still works for the case $\nu \in (\frac{1}{4},\frac{1}{3}]$ with an extension of Gubinelli derivative to the second order, although the computation would be rather complicated. Moreover, it could also be applied for proving the general case in which $g$ is unbounded, even though we then need to prove the existence and uniqueness theorem first. The reader is referred to \cite{lejay} and \cite{coutinlejay} for this approach, in which the differential equation is understood in the sense of Davie \cite{davie}.

\section{Rough differential equations}

We would like to give a brief introduction to Young integrals. Given any compact time interval $I \subset \R$, let $C(I,\R^d)$ denote the space of all continuous paths $y:\;I \to \R^d$ equipped with sup norm $\|\cdot\|_{\infty,I}$ given by $\|y\|_{\infty,I}=\sup_{t\in I} \|y_t\|$, where $\|\cdot\|$ is the Euclidean norm in $\R^d$. We write $y_{s,t}:= y_t-y_s$. For $p\geq 1$, denote by $\cC^{p{\rm-var}}(I,\R^d)\subset C(I,\R^d)$ the space of all continuous path $y:I \to \R^d$ which is of finite $p$-variation 
\begin{eqnarray}
\ltn y\rtn_{p\text{-var},I} :=\left(\sup_{\Pi(I)}\sum_{i=1}^n \|y_{t_i,t_{i+1}}\|^p\right)^{1/p} < \infty,
\end{eqnarray}
where the supremum is taken over the whole class of finite partition of $I$. $\cC^{p{\rm-var}}(I,\R^d)$ equipped with the $p-$var norm
\begin{eqnarray*}
	\|y\|_{p\text{-var},I}&:=& \|y_{\min{I}}\|+\ltn y\rtn_{p\rm{-var},I},
\end{eqnarray*}
is a nonseparable Banach space \cite[Theorem 5.25, p.\ 92]{friz}. Also for each $0<\alpha<1$, we denote by $C^{\alpha}(I,\R^d)$ the space of H\"older continuous functions with exponent $\alpha$ on $I$ equipped with the norm
\[
\|y\|_{\alpha,I}: = \|y_{\min{I}}\| + \ltn y\rtn_{\alpha,I}=\|y(a)\| + \sup_{s<t\in I }\frac{\|y_{s,t}\|}{(t-s)^\alpha},
\]
A continuous map $\overline{\omega}: \Delta^2(I)\longrightarrow \R^+, \Delta^2(I):=\{(s,t): \min{I}\leq s\leq t\leq \max{I}\}$ is called a {\it control} if it is zero on the diagonal and superadditive, i.e.  $\overline{\omega}_{t,t}=0$ for all $t\in I$, and  $\overline{\omega}_{s,u}+\overline{\omega}_{u,t}\leq \overline{\omega}_{s,t}$ for all $s\leq u\leq t$ in $I$.\\
Now, consider $y\in \cC^{q{\rm-var}}(I,\cL(R^m,\R^d))$ and $x\in \cC^{p{\rm -var}}(I,\R^m)$ with  $\frac{1}{p}+\frac{1}{q}  > 1$, the Young integral $\int_I y_t d x_t$ can be defined as 
\[
\int_I y_s d x_s:= \lim \limits_{|\Pi| \to 0} \sum_{[u,v] \in \Pi} y_u x_{u,v} ,
\]
where the limit is taken on all the finite partition $\Pi=\{ \min{I} =t_0<t_1<\cdots < t_n=\max{I} \}$ of $I$ with $|\Pi| := \displaystyle\max_{[u,v]\in \Pi} |v-u|$ (see \cite[p.\ 264--265]{young}). This integral satisfies additive property by the construction, and the so-called Young-Loeve estimate \cite[Theorem 6.8, p.\ 116]{friz}
\begin{eqnarray}\label{YL0}
\Big\|\int_s^t y_u d x_u-y_s x_{s,t}\Big\| &\leq& K(p,q) \ltn y\rtn_{q\text{-var},[s,t]} \ltn x\rtn_{p\text{-var},[s,t]} \notag\\
&\leq& K(p,q) |t-s|^{\frac{1}{p}+\frac{1}{q}} \ltn y \rtn_{\frac{1}{p},[s,t]} \ltn x \rtn_{\frac{1}{q}{\rm -Hol},[s,t]}, 
\end{eqnarray}
for all $[s,t]\subset I$, where 
\begin{equation}\label{constK}
K(p,q):=(1-2^{1-\frac{1}{p} - \frac{1}{q}})^{-1}.
\end{equation}

We also introduce the construction of the integral using rough paths for the case $y,x \in C^\beta(I)$ when $\beta\in(\frac{1}{3},\nu)$. To do that, we need to introduce the concept of rough paths. Following \cite{frizhairer}, a couple $\bx=(x,\X)$, with $x \in C^\beta(I,\R^m)$ and $\X \in C^{2\beta}_2(\Delta^2(I),\R^m \otimes \R^m):= \{\X: \sup_{s<t} \frac{\|\X_{s,t}\|}{|t-s|^{2\beta}} < \infty \}$ where the tensor product $\R^m \otimes \R^n$ can be indentified with the matrix space $\R^{m\times n}$, is called a {\it rough path} if they satisfies Chen's relation
\begin{equation}\label{chen}
\X_{s,t} - \X_{s,u} - \X_{u,t} = x_{u,t} \otimes x_{s,u},\qquad \forall \min{I} \leq s \leq u \leq t \leq \max{I}. 
\end{equation}
$\X$ is viewed as {\it postulating} the value of the quantity $\int_s^t x_{s,r} \otimes dx_r := \X_{s,t}$ where the right hand side is taken as a definition for the left hand side. Denote by $\cC^\beta(I) \subset C^\beta \oplus C^{2\beta}_2$ the set of all rough paths in $I$, then $\cC^\beta$ is a closed set but not a linear space, equipped with the rough path semi-norm 
\begin{equation}\label{translated}
\ltn \bx \rtn_{\beta,I} := \ltn x \rtn_{\beta,I} + \ltn \X \rtn_{2\beta,\Delta^2(I)}^{\frac{1}{2}} < \infty.  
\end{equation}
Let $3>p> 2, \nu > \frac{1}{p}$. Throughout this paper, we will assume that  $x(\omega): I \to \R^m$ and $\X(\omega): I \times I \to \R^m \otimes \R^m$ are random funtions that satisfy Chen's relation relation \eqref{chen} and
\begin{equation}\label{expect}
\Big(E \|x_{s,t} \|^p\Big)^{\frac{1}{p}} \leq C |t-s|^\nu, \quad \text{and} \quad 	\Big(E \|\X_{s,t}\|^{\frac{p}{2}}\Big)^{\frac{2}{p}} \leq C |t-s|^{2\nu},\forall s,t \in I
\end{equation}
for some constant $C$. Then, due to the Kolmogorov criterion for rough paths \cite[Appendix A.3]{friz} for all $\beta\in (\frac{1}{3},\nu)$ there is a version of $\omega-$wise $(x,\X)$ and random variables $K_\beta \in L^p, \K_\beta \in L^{\frac{p}{2}}$, such that, $\omega-$wise speaking, for all $s,t \in I$,
\[
\|x_{s,t}\| \leq K_\alpha |t-s|^\beta, \quad \|\X_{s,t}\| \leq \K_\beta |t-s|^{2\beta},
\] 
so that $(x,\X) \in \cC^\beta.$ Moreover, we could choose $\beta$ such that 
\begin{eqnarray*}
&&x \in C^{0,\beta}(I):= \{x \in C^\beta: \lim \limits_{\delta \to 0}\sup_{0<t-s <\delta} \frac{\|x_{s,t}\|}{|t-s|^\beta} = 0\}, \\
&&\X \in C^{0,2\beta}(\Delta^2(I)):= \{ \X \in C^{2\beta}(\Delta^2(I)): \lim \limits_{\delta \to 0}\sup_{0<t-s <\delta} \frac{\|\X_{s,t}\|}{|t-s|^{2\beta}} = 0  \}, 
\end{eqnarray*}
then $\cC^{0,\beta}(I) \subset C^{0,\beta}(I) \oplus C^{0,2\beta}(\Delta^2(I))$ is separable due to the separability of $C^{0,\beta}(I)$ and $C^{0,2\beta}(\Delta^2(I))$. 

\subsection{Controlled rough paths}

A path $y \in C^\beta(I,\cL(\R^m,\R^d))$ is then called to be {\it controlled by} $x \in C^\beta(I,\R^m)$ if there exists a tube $(y^\prime,R^y)$ with $y^\prime \in C^\beta(I,\cL(\R^m,\cL(\R^m,\R^d))), R^y \in C^{2\beta}(\Delta^2(I),\cL(\R^m,\R^d))$ such that
\[
y_{s,t} = y^\prime_s x_{s,t} + R^y_{s,t},\qquad \forall \min{I}\leq s \leq t \leq \max{I}.
\]
$y^\prime$ is called Gubinelli derivative of $y$, which is uniquely defined as long as $x \in C^\beta\setminus C^{2\beta}$ (see \cite[Proposition 6.4]{frizhairer}). The space $\cD^{2\beta}_x(I)$ of all the couple $(y,y^\prime)$ that is controlled by $x$ will be a Banach space equipped with the norm
\begin{eqnarray*}
	\|y,y^\prime\|_{x,2\beta,I} &:=& \|y_{\min{I}}\| + \|y^\prime_{\min{I}}\| + \ltn y,y^\prime \rtn_{x,2\beta,I},\qquad \text{where} \\
	\ltn y,y^\prime \rtn_{x,2\beta,I} &:=& \ltn y^\prime \rtn_{\beta,I} +   \ltn R^y\rtn_{2\beta,I},
\end{eqnarray*}
where we omit the value space for simplicity of presentation. Now fix a rough path $(x,\X)$, then for any $(y,y^\prime) \in \cD^{2\beta}_x (I)$, it can be proved that the function $F \in C^\beta(\Delta^2 (I),\R^d)$ defined by 
\[
F_{s,t} := y_s x_{s,t} + y^\prime_s \X_{s,t}
\] 
belongs to the space 
\begin{eqnarray*}
	C^{\beta, 3\beta}_2(I) &:=& \Big \{ F\in C^\beta(\Delta^2(I)): F_{t,t} =0 \quad \text{and}\\ && \quad \qquad \qquad \qquad \qquad \ltn \delta F \rtn_{3\beta,I} := \sup_{\min{I} \leq s \leq u \leq t \leq \max{I}} \frac{\|F_{s,t} - F_{s,u}-F_{u,t}\|}{|t-s|^{3\beta}} < \infty \Big\}.
\end{eqnarray*}
Thanks to the sewing lemma \cite[Lemma 4.2]{frizhairer}, the integral $\int_s^t y_u dx_u$ can be defined as  
\[
\int_s^t y_u dx_u := \lim \limits_{|\Pi| \to 0} \sum_{[u,v] \in \Pi} [ y_{u}x_{u,v} + y^\prime_u \X_{u,v} ]
\]
where the limit is taken on all the finite partition $\Pi$ of $I$ with $|\Pi| := \displaystyle\max_{[u,v]\in \Pi} |v-u|$ (see \cite{gubinelli}). Moreover, there exists a constant $C_\beta = C_{\beta,|I|} >1$ with $|I| := \max{I} - \min{I}$, such that
\begin{equation}\label{roughEst}
\Big\|\int_s^t y_u dx_u - y_s x_{s,t} + y^\prime_s \X_{s,t}\Big\| \leq C_\beta |t-s|^{3\beta} \Big(\ltn x \rtn_{\beta,[s,t]} \ltn R^y \rtn_{2\beta,\Delta^2[s,t]} + \ltn y^\prime\rtn_{\beta,[s,t]} \ltn \X \rtn_{2\beta,\Delta^2[s,t]}\Big).
\end{equation}
From now on, if no other emphasis, we will simply write $\ltn x \rtn_{\beta}$ or $\ltn \X \rtn_{2\beta}$ without addressing the domain in $I$ or $\Delta^2(I)$. In particular, for any $f \in C^3_b(\R^d,\R^d)$ we get the formula for integration by composition
\[
f(x_t) = f(x_s) + \int_s^t \nabla f(x_u) dx_u + \frac{1}{2} \int_s^t \nabla^2f(x_u) d[x]_{s,u},
\]
where the last integral is understood in the Young sense and $[x]_{s,t}:= x_{s,t} \otimes x_{s,t}  - 2 \text{\ Sym\ } (\X_{s,t}) \in C^{2\beta}$. Notice that for geometric rough path $\X_{s,t} = \int_s^t x_{s,r} \otimes dx_r$, then $\text{\ Sym\ } (\X_{s,t}) = \frac{1}{2} x_{s,t} \otimes x_{s,t}$, thus $[x]_{s,t} \equiv 0.$\\
The following lemma is from \cite{duc19part1}.
\begin{lemma}[Change of variables formula]
	Assume that $\beta > \frac{1}{3}$, $V \in C^3_b(\R^d,\R)$ and $y \in C^{\beta}(I,\R) $ is a solution of the rough differential equation
	\begin{equation}\label{roughde1}
	y_t = y_s + \int_s^t f(y_u)du + \int_s^t g(y_u)dx_u,\quad \forall \min{I} \leq s \leq t \leq \max{I}.
	\end{equation}
	Then one get the change of variable formula
	\begin{eqnarray}\label{Roughformula}
	V(y_t) &=& V(y_s) + \int_s^t \langle D_y V(y_u), f(y_u)\rangle  du + \int_s^t \langle D_y V(y_u) g(y_u) \rangle d x_u \notag \\
	&& + \frac{1}{2} \int_s^t D_{yy}V(y_u) [g(y_u),g(y_u)] d[x]_{s,u},  
	\end{eqnarray}
	where
	\[
	[ D_y V(y) g(y)]^\prime_s = \langle D_y V(y_s), D_y g(y_s)g(y_s)\rangle + D_{yy}V(y_s)[g(y_s),g(y_s)].
	\]
\end{lemma}
In practice, we would use the $p$-var norm
\begin{eqnarray*}
	\|y,y^\prime\|_{x,p,I} &:=& \|y_{\min{I}}\| + \|y^\prime_{\min{I}}\| + \ltn y,y^\prime \rtn_{x,p,I},\qquad \text{where} \\
	\ltn y,y^\prime \rtn_{x,p,I} &:=& \ltn y^\prime \rtn_{p{\rm -var},I} +   \ltn R^y\rtn_{\frac{p}{2}{\rm -var},I}.
\end{eqnarray*}
Thanks to the sewing lemma \cite{coutin}, we can use a similar version to \eqref{roughEst} under $p-$var norm as follows.
\begin{equation}\label{roughpvar}
\Big\|\int_s^t y_u dx_u - y_s x_{s,t} + y^\prime_s \X_{s,t}\Big\| \leq C_\beta \Big(\ltn x \rtn_{p{\rm - var},[s,t]} \ltn R^y \rtn_{\frac{p}{2}{\rm -var},\Delta^2[s,t]} + \ltn y^\prime\rtn_{p{\rm - var},[s,t]} \ltn \X \rtn_{\frac{p}{2}{\rm -var},\Delta^2[s,t]}\Big).
\end{equation}

\subsection{Greedy times and integrability}
In this part, we would like to develop a modified version of greedy times as in \cite{cassetal}, for which we need a little more regularity. Given fixed $\nu \in (\frac{1}{3},\frac{1}{2}), \frac{1}{p} \in (\frac{1}{3},\nu), \sigma \in (0,\nu-\frac{1}{p})$ and $\beta = \frac{1}{p}+\sigma$, on each compact interval $I$ such that $|I|=\max{I} - \min{I} \leq 1$, consider a rough path $\bx = (x,\X) \in \cC^{\tp}(I)$ with the modified $(\tp)$ - norm $\ltn \bx \rtn_{p,\sigma} := \ltn x \rtn_{p,\sigma} + \ltn \X \rtn_{q,\sigma}^{\frac{1}{2}}$ defined by 
\begin{equation}\label{psigmanorm}
 \ltn x \rtn_{p,\sigma} = \Big(\sup_{\Pi} \sum_{[u,v] \in \Pi} \|x_{u,v}\|^p |v-u|^{-\sigma p}\Big)^{\frac{1}{p}}, \quad \ltn \X \rtn_{q,\sigma} = \Big(\sup_{\Pi} \sum_{[u,v] \in \Pi} \|\X_{u,v}\|^q |v-u|^{-\sigma q}\Big)^{\frac{1}{q}},
\end{equation}
where $q=\frac{p}{2}$. The following lemma is easy to prove.

\begin{lemma}
$\ltn \X \rtn_{\tq}^q$ and $\ltn x \rtn_{\tp}^p$ are control functions. In addition, $\cC^\beta(I) \subset \cC^{\tp}(I) \subset \cC^{p{\rm -var}}(I)$ and for $\bx \in \cC^\beta(I) $ we have the estimates
\begin{equation}
|I|^{-\sigma} \ltn x \rtn_{p{\rm -var},I} \leq \ltn x \rtn_{\tp,I} \leq |I|^{\frac{1}{p}}\ltn x \rtn_{\frac{1}{p}+\sigma,I}; \quad |I|^{-\sigma} \ltn \X \rtn_{q{\rm -var}, I} \leq \ltn \X \rtn_{\tq,I} \leq |I|^{\frac{1}{q}}\ltn \X \rtn_{\frac{1}{q}+\sigma,I}.
\end{equation}
\end{lemma}

Given $\frac{1}{p}\in (\frac{1}{3},\nu)$ and $\sigma \in (0,\nu-\frac{1}{p})$, we construct for any fixed $\gamma \in (0,1)$ the sequence of greedy times $\{\tau_i(\gamma,I,\tp)\}_{i \in \N}$ w.r.t. H\"older norms 
\begin{equation}\label{greedytime}
\tau_0 = \min{I},\quad \tau_{i+1}:= \inf\Big\{t>\tau_i:  \ltn \bx \rtn_{\tp, [\tau_i,t]} = \gamma \Big\}\wedge \max{I}.
\end{equation}
Denote by $N_{\gamma,I,\tp}(\bx):=\sup \{i \in \N: \tau_i \leq \max{I}\}$. Also, we construct another sequence of greedy time $\{\bar{\tau}_i(\gamma,I,\tp)\}_{i \in \N}$ given by
\begin{equation}\label{greedywitht}
\bar{\tau}_0 = \min{I},\quad \bar{\tau}_{i+1}:= \inf\Big\{t>\bar{\tau}_i:  (t-\bar{\tau}_i)^{\sigma} + \ltn \bx \rtn_{\tp, [\bar{\tau}_i,t]} = \gamma \Big\}\wedge \max{I},
\end{equation}
and denote by $\bar{N}_{\gamma,I,\tp}(\bx):=\sup \{i \in \N: \bar{\tau}_i \leq \max{I}\}$. Then on any interval $J$ such that $|J| = \Big(\frac{\gamma}{2}\Big)^{\frac{1}{\sigma}}$ and with the sequence $\{\tau_i(\frac{\gamma}{2},J,\tp)\}_{i\in\N}$ it follows that
\[
(\tau_{i+1} - \tau_i)^{\frac{1}{\sigma}} + \ltn \bx \rtn_{\tp, [\tau_i,\tau_{i+1}]} \leq \frac{\gamma}{2} + \frac{\gamma}{2} = \gamma, 
\]
hence there is a most one greedy time of the sequence $\bar{\tau}_i$ lying in each interval $[\tau_i(\frac{\gamma}{2},J,\tp), \tau_{i+1}(\frac{\gamma}{2},J,\tp)]$. That being said, if we divide $I$ into sub-interval $J_k$ of length $|J_k| \equiv |J|= \Big(\frac{\gamma}{2}\Big)^{\frac{1}{\sigma}}$, then it follows that
\begin{equation}
\bar{N}_{\gamma,I,\tp}(\bx) \leq \sum_{k=1}^{m} N_{\frac{\gamma}{2},J_k,\tp}(\bx),\quad m := \Big\lceil \frac{|I|}{|J|} \Big\rceil.
\end{equation}
We need to show that $\exp\{N_{\gamma,I,\tp}(\bx)\}$	is also integrable for any interval $I$ such that $|I| \leq 1$.

\subsubsection*{Translated rough paths}
Given $q = \frac{p}{2}$, $\frac{1}{p} \in (\frac{1}{3},\nu) $ and $\sigma \in (0, \nu-\frac{1}{p})$ then $\frac{1}{p}+ \frac{1}{q} >1$. Following \cite[Chapter 10 \& Chapter 11]{frizhairer}, let  $\mathcal{W} = C(I,\R^m)$ be the probability space equipped with a Gaussian measure $\mathbb{P}$ and let $(X_t)$ be a continuous, mean zero Gaussian process, parameterized over a compact interval $I$. The associated {\it Cameron-Martin space} $\mathcal{H} \subset \mathcal{W}$ consists of paths $t\mapsto h_\cdot = E(Z X_\cdot)$ where $Z \in \mathcal{W}^1$ is an element in the so-called {\it first Wiener chaos}. If $\bar{h}_\cdot = E(\bar{Z}X_\cdot)$ denotes another element in $\mathcal{H}$ then the inner product $\langle h,\bar{h} \rangle_\mathcal{H} := E (Z \bar{Z})$ makes $\mathcal{H}$ a Hilbert space and $Z\mapsto h$ is an isometry between $\mathcal{W}^1$ and $\mathcal{H}$. The triple $(\mathcal{W}, \mathcal{H},\mathbb{P})$ is then called the {\it abstract Wiener space}. \\
We need a little more regularity for the rectangular increments of the covariance 
\[
R \Big(\begin{array}{cc} s&t \\ s^\prime & t^\prime \end{array}\Big)  := E(X_{s,t} \otimes X_{s^\prime,t^\prime}) 
\]
which seems to be natural for Gaussian processes with stationary increments.
\begin{eqnarray}\label{covariance}
\ltn R\rtn_{\tq,I^2} &<& \infty,\quad \text{where} \quad \\
\ltn R \rtn_{\tq, I \times I^\prime} &:=& \Big( \sup_{\Pi(I), \Pi(I^\prime)} \sum_{[s,t]\in \Pi(I), [s,t] \in \Pi(I^\prime)} \Big| R \Big(\begin{array}{ccc}
s & t \\ s^\prime & t^\prime 
\end{array}\Big)\Big|^{q} |t-s|^{-\sigma q} |t^\prime - s^\prime|^{-\sigma q} \Big)^{\frac{1}{q}}. \notag
\end{eqnarray}
Given \eqref{covariance}, we prove a modified version of \cite[Proposition 11.2]{frizhairer} that $\mathcal{H}$ is continuously embedded in the space of continuous paths of finite $(\tq)$-variation, i.e. $\mathcal{H} \hookrightarrow C^{\tq}(I,\R^d)$, and there exists a constant $C_{\rm emb} >0$  such that for all $h \in \mathcal{H}$ and all $s<t$ in $I$,
\[
\ltn h \rtn_{\tq,[s,t]} \leq \|h\|_{\mathcal{H}} \sqrt{\ltn R\rtn_{\tq,[s,t]^2}} \leq C_{\rm emb} \|h\|_{\mathcal{H}}.
\]
The proof goes line in line with the one of  \cite[Proposition 11.2]{frizhairer} except that we need to add terms $|t_{j+1}-t_j|^{-\sigma q}$ and $|t_{k+1}-t_k|^{-\sigma q}$ in the expression of elements in $l^q$ and its dual space $l^{q^\prime}$, where $\frac{1}{q} + \frac{1}{q^\prime} =1$. \\ 
That means $h \in C^{q-{\rm var}}(I,\R^m)$ is of complementary Young regularity, but "respecful" of $\sigma$-H\"older regularity in the sense that $\ltn h \rtn_{\tq,I} < \infty$. It then makes sense (see e.g. \cite{friz} or \cite{cassetal}) to define the so-called {\it translated rough path} $T_h\bx$ as
\[
T_h \bx := \Big(x + h, \X + \int h \otimes dx + \int x \otimes dh + \int h \otimes dh\Big).
\]
We are going to prove that
\begin{lemma}
	Given $|I| \leq 1$, the translated map $T_h: \cC^{\tp} \to \cC^{\tp}$ such that
	for any $[s,t] \subset I$ we have the estimate
	\begin{equation}\label{translatedest}
	\ltn T_h \bx \rtn_{\tp, [s,t]} \leq K(\tp) \Big(\ltn \bx \rtn_{\tp,[s,t]} + |t-s|^{\frac{\sigma}{2}}\ltn h \rtn_{\tq,[s,t]}\Big)
	\end{equation} 
\end{lemma}

\begin{proof}
	The proof is quite direct and similar to \cite[Lemma 3.1]{cassetal}. By assigning $K:= (1-2^{1-\frac{3}{p}})^{-1}$ observe that
	\allowdisplaybreaks
	\begin{eqnarray*}
		\ltn T_h \bx \rtn_{\tp,J} &\leq& \ltn x+h \rtn_{\tp,J} + \Big(\ltn \X \rtn_{\tq,J} + \ltn \int h \otimes dx \rtn_{\tq,J}+ \ltn \int x \otimes dh \rtn_{\tq,J}+ \ltn \int h \otimes dh \rtn_{\tq,J}\Big)^{\frac{1}{2}}\\
		&\leq& \ltn x\rtn_{\tp,J} + |J|^{\sigma} \ltn h \rtn_{\tq,J} + \ltn \X \rtn_{\tq,\Delta^2(J)}^{\frac{1}{2}} \\
		&&+ \ltn \int h \otimes dx \rtn_{\tq,\Delta^2(J)}^{\frac{1}{2}} 
		+\ltn \int x \otimes dh \rtn_{\tq,\Delta^2(J)}^{\frac{1}{2}} + \ltn \int h \otimes dh \rtn_{\tq,\Delta^2(J)}^{\frac{1}{2}}\\
		&\leq& \ltn \bx\rtn_{\tp,J} + |J|^\sigma \ltn h \rtn_{\tq,J} + 2 K^{\frac{1}{2}} |J|^\frac{\sigma}{2} \ltn x \rtn_{\tp,J}^{\frac{1}{2}} \ltn h \rtn_{\tq,J}^{\frac{1}{2}} + K^{\frac{1}{2}} |J|^{\frac{\sigma}{2}} \ltn h \rtn_{\tq,J}\\
		&\leq&  \ltn \bx\rtn_{\tp,J} + |J|^\frac{\sigma}{2} \ltn h \rtn_{\tq,J} + K^{\frac{1}{2}} (\ltn \bx \rtn_{\tp,J} + |J|^\frac{\sigma}{2} \ltn h \rtn_{\tq,J}) +  K^{\frac{1}{2}} |J|^\frac{\sigma}{2} \ltn h \rtn_{\tq,J}\\
		&\leq& (1 + 2K^{\frac{1}{2}}) (\ltn \bx\rtn_{\tp,J} +|J|^\frac{\sigma}{2} \ltn h \rtn_{\tq,J}).
	\end{eqnarray*}
	Hence \eqref{translatedest} holds by assigning $K(\tp) := 1 + 2K^{\frac{1}{2}}$.
\end{proof}
	
\begin{theorem}[Tail estimate and integrability]\label{integrability}
	 Assume that $X$ has a natural lift to a geometric $(\tp)$- variation rough path $\mathbf{X}$ and there exists $C_{\rm emb} \in (0,\infty)$ with $\ltn h \rtn_{\tq,I} \leq C_{\rm emb} \|h\|_{\mathcal{H}}$ for all $h \in \mathcal{H}$.  Then for a fixed $I$ with $|I|\leq 1$, there exists a set $E\subset \mathcal{W}$ of $\mathbb{P}$-full measure, with the property: for all $\omega \in E, h \in \mathcal{H}$ and $\gamma >0$, if 
	\begin{equation}\label{Nest}
	\ltn \mathbf{X}(\omega-h) \rtn_{\tp,I} \leq \gamma\quad \text{then} \quad	
	|I|^{\frac{q\sigma}{2}}  \ltn h \rtn_{\tq,I}^q \gamma^{-q\sigma} \geq N_{2K(\tp) \gamma,I,\tp}(\mathbf{X}(\omega)).
	\end{equation}
	Moreover, 
	\begin{equation}\label{tailest}
	\mathbb{P} \{ \omega:N_{2K(\tp) \gamma,I,\tp}(\mathbf{X}(\omega))>n \} \leq \exp \Big\{2 \Phi^{-1}(\mathbb{P}(B_\gamma))^2\Big\}  \exp \Big\{ \frac{-\gamma^2 n^{\frac{2}{q}}}{2C^2_{\rm emb} |I|^{\frac{4}{\sigma}}}\Big\},
	\end{equation} 
where $\Phi^{-1}$ is the inverse of the standard normal cumulative distribution function and $B_\gamma:= \{\omega \in \mathcal{W}: \ltn \mathbf{X}(\omega) \rtn_{\tp,I} \leq \gamma\}$. In particular, $\exp \{N_{2K(\tp) \gamma,I,\tp}(\mathbf{X}(\omega))\}$ is integrable.
\end{theorem}
\begin{proof}
	We follows the arguments in \cite[Proposition 6.2 \& Theorem 6.3]{cassetal} line by line. From the definition of the sequence $\tau_i$ and the integer $N_{2K(\tp) \gamma,I,\tp}(\mathbf{X}(\omega))$ we have $\ltn \mathbf{X}(\omega) \rtn_{\tp,[\tau_i,\tau_{i+1}]} = 2K(\tp) \gamma$. Consider $E := \{ \omega \in \mathcal{W}: T_h \mathbf{X}(\omega) = \mathbf{X}(\omega + h), \forall h \in \mathcal{H}\}$ then $\mathbb{P}(E) = 1$ by \cite[Theorem 11.5]{frizhairer} or \cite[Lemma 5.4]{cassetal} or \cite[Lemma 15.58]{friz}. For every $\omega \in E$ and $h \in F_{\omega,\gamma} := \{h \in \mathcal{H}: \ltn \mathbf{X}(\omega-h) \rtn_{\tp,I} \leq \gamma \}$, using \eqref{translatedest} we have
	\begin{eqnarray*}
		2K(\tp) \gamma&=&\ltn \mathbf{X}(\omega) \rtn_{\tp,[\tau_i,\tau_{i+1}]} = \ltn T_h \mathbf{X}(\omega-h) \rtn_{\tp,[\tau_i,\tau_{i+1}]}\\
		& \leq& K(\tp) \Big(\ltn \mathbf{X}(\omega -h) \rtn_{\tp,[\tau_i,\tau_{i+1}]} + |\tau_{i+1}-\tau_i|^{\frac{\sigma}{2}}\ltn h \rtn_{\tq,[\tau_i,\tau_{i+1}]} \Big)\\
		&\leq&   K(\tp) \gamma +  K(\tp) |\tau_{i+1}-\tau_i|^{\frac{\sigma}{2}}\ltn h \rtn_{\tq,[\tau_i,\tau_{i+1}]},
	\end{eqnarray*}
	which leads to $|\tau_{i+1}-\tau_i|^{\frac{\sigma}{2}} \ltn h \rtn_{\tq,[\tau_i,\tau_{i+1}]} \geq \gamma$ and
	\[
	|I|^{\frac{q\sigma}{2}} \ltn h \rtn^q_{\tq,[\tau_i,\tau_{i+1}]}\geq \gamma^{q}.
	\] 
	Hence using the fact that $\ltn h \rtn^q_{\tq}$ is a control function, by taking the summation on all possible interval $[\tau_i,\tau_{i+1}]$, we get
	\[
	N_{2K(\tp) \gamma,I,\tp}(\mathbf{X}(\omega)) \gamma^q\leq |I|^{\frac{q\sigma}{2}} \sum_{i=0}^{N_{2K_\alpha \gamma,I,\tp}(\mathbf{X}(\omega))-1} \ltn h \rtn_{\tq,\nu,[\tau_i,\tau_{i+1}]}^p \leq |I|^{\frac{q\sigma}{2}} \ltn h \rtn_{\tq,I}^q \leq |I|^{\frac{q\sigma}{2}} C_{\rm emb}^q\| h \|_{\mathcal{H}}^q, 
	\]
	which follows \eqref{Nest}. As a result, 
	\[
	\{\omega: N_{2K(\tp) \gamma,I,\tp}(\mathbf{X}(\omega)) > n\} \cap E \subset \mathcal{W} \setminus (B_\gamma + r_n \mathcal{K})
	\]
	where $\mathcal{K}$ denotes the unit ball in $\mathcal{H}$, $B_\gamma +r_n \mathcal{K}:= \{x+r_ny : x \in B_\gamma, y \in \mathcal{K}\}$ is the Minkowski sum, and $r_n := \frac{\gamma n^{\frac{2}{q}}}{C_{\rm emb} |I|^{\frac{2}{\sigma}}}$. The rest applies Borell's inequality as  in \cite[Theorem 6.1 \& Theorem 6.3]{cassetal} so that 
	\[
	\mathbb{P} \{ \omega:N_{2K(\tp) \gamma,I,\tp}(\mathbf{X}(\omega))>n \} \leq \exp (2 b_\gamma^2) \exp \Big(-\frac{r_n^2}{2} \Big),
	\]
	where $\mathbb{P}(B_\gamma)=: \Phi(b_\gamma)$. This proves \eqref{tailest} and the integrability of $\exp \{N_{2K(\tp) \gamma,I,\tp}(\mathbf{X}(\omega))\}$ (see also \cite[Remark 6.4]{cassetal}.
\end{proof}

\begin{corollary}\label{Nintegrability}
	For $|I| \leq 1$, then $\exp\{\bar{N}_{\gamma,I,\tp}(\bx)\}$ is integrable. Moreover, there exists a limit
	\begin{equation}\label{Nbarergodic}
	\lim \limits_{n \to \infty} \frac{1}{n}\sum_{k = 0}^{n-1}P\Big(\exp \Big \{\Lambda \bar{N}_{\frac{\mu}{2M},[k,k+1],\tp}(\bx)\Big\}\Big)  = E P \Big(\exp \Big \{\Lambda \bar{N}_{\frac{\mu}{2M},[0,1],\tp}(\bx)\Big\}\Big) < \infty.
	\end{equation}
\end{corollary}
\begin{proof}
	The conclusion follows directly from the integrability of $\exp \Big\{m N_{\frac{\gamma}{2},J_k,\tp}(\bx)\Big\}$ and the Cauchy inequality that
	\[
	\exp\{\bar{N}_{\gamma,I,\tp}(\bx)\} \leq \prod_{k=1}^{m} \exp\{N_{\frac{\gamma}{2},J_k,\tp}(\bx)\} \leq \frac{1}{m} \sum_{k=1}^{m} \exp \Big\{m N_{\frac{\gamma}{2},J_k,\tp}(\bx)\Big\},\quad m = \Big\lceil \frac{|I|}{|J|} \Big\rceil.
	\]
	Since $\mathbf{X}$ also generates a rough cocycle \cite{BRSch17}, it it easy to prove that
	\[
	\ltn \bx(\theta_a \omega) \rtn_{\tp,[s,t]} = \ltn \bx(\omega) \rtn_{\tp,[s+a,t+a]},   
	\]
	so that $\bar{N}_{\gamma,[k,k+1],\tp}(\bx) = \bar{N}_{\gamma,[0,1],\tp}(\bx(\theta_k \omega))$. \eqref{Nbarergodic} is then followed from the ergodic Birkhorff theorem.
\end{proof}	

\subsection{Existence, uniqueness and integrability of the solution}
\begin{theorem}[Existence and uniqueness of the solution]
	Under the mild assumptions, there exists a unique solution of equation \eqref{RDE1} and also of the backward equation on any interval $[a,b]$. 
\end{theorem}
\begin{proof}
	Since there are similar versions for $p$ - variation norm in \cite{gubinelli} and \cite{riedelScheutzow}, we would only sketch out the proof here. We first solve the rough differential equation
	\begin{equation}\label{roughg}
	dz = g(z) dx_t.
	\end{equation}
	From \cite{gubinelli}, we could apply Schauder-Tichonorff theorem to conclude that there exists a unique solution of \eqref{roughg} on $\cD_x^{2\beta}([a,b])$ where $\beta = \frac{1}{p}+\sigma$. Moreover, denote $\varphi(t,x,z_a)= z_t$ to be the solution mapping of \eqref{roughg} then we can prove that $\varphi$ is $C^1$ w.r.t. $z_a$ and in the $\ltn \cdot, \cdot \rtn_{x, \tp}$ - norm. More specifically, by using Lemmas \ref{roughintegral}, \ref{roughfunction}, \ref{roughfunctiondiff} and the greedy time sequence $\{\bar{\tau}_i(\frac{\mu}{2M},[a,b],\tp)\}_{i\in \N}$ in \eqref{greedywitht}, where $\mu \in (0,1)$ is fixed and $M \geq \frac{1}{2}$ is a constant dependent of $C_g,\alpha$, we can prove that there exists a generic constant $\Lambda = \Lambda ([a,b],z_a,C_g,\bx)$ such that
	\begin{eqnarray}\label{continuity}
	\ltn (\bar{z},\bar{z}^\prime) -(z,z^\prime)\rtn_{x,\tp,[a,b]} &\leq& \|\bar{z}_a-z_a\|\exp \Big\{\Lambda \bar{N}_{\frac{\mu}{2M},[a,b],\tp}(\bx)\Big \} \notag \\
	\|\bar{z} - z\|_{\infty,[a,b]} &\leq&  \Lambda \|\bar{z}_a-z_a\| \Big(1+\exp \Big\{\Lambda \bar{N}_{\frac{\mu}{2M},[a,b],\tp}(\bx)\Big \} \Big)
	\end{eqnarray}
	In fact denote by $\Phi(t,x,z_a)$ the solution matrix of the time dependent linearized system
	\[
	d\xi_t = D_z g(\varphi(t,x,z_a)) \xi_t dx_t,
	\]
	then $\xi = \Phi(t,x,z_a)(\bar{z}_a-z_a)$ is the solution of the linearized system given initial point $\xi_a =\bar{z}_a-z_a$. Assign $r_t := \bar{z}_t - z_t - \xi_t$, then $r_a =0$ and
	\begin{eqnarray}\label{equationrt}
	r_t &=& \int_a^t \Big[\int_0^1 D_zg(z_s + \eta(\bar{z}_s-z_s)) - D_zg(z_s)\Big](\bar{z}_s - z_s)d\eta dx_s + \int_a^t D_zg(z_s) r_s dx_s,\notag\\
	&=& e_t+\int_a^t D_zg(z_s) r_s dx_s, \quad \forall t \in [a,b],
	\end{eqnarray}	
	where $e$ is also controlled by $x$ with $e_a =0$ and 
	\[
	\|e^\prime_a\| \leq \int_0^1 \|D_z g(z_a+ \eta (\bar{z}_a - z_a)) - D_z g(z_a) \| \|\bar{z}_a - z_a\| d\eta\leq \frac{1}{2}C_g \|\bar{z}_a - z_a\|^2.
	\]
	From \eqref{equationrt} it can be proved that 
	\begin{eqnarray}\label{residualdiff}
	\|r\|_{\infty,[a,b]} \vee \ltn r,r^\prime \rtn_{x,\tp,[a,b]} &\leq& \Big( \|e^\prime_a\| + \ltn e,e^\prime\rtn_{x,\tp,[a,b]}\Big) \exp \Big\{\Lambda \bar{N}_{\frac{\mu}{2M},[a,b],\tp}(\bx)\Big \} \notag\\
	&\leq& \Lambda \|\bar{z}_a-z_a\|^2 \exp \Big\{\Lambda \bar{N}_{\frac{\mu}{2M},[a,b],\tp}(\bx)\Big \},
	\end{eqnarray}
	which proves $\varphi(t,x,z_a)$ to be $C^1$ w.r.t. $z_a$, with corresponding derivative $\Phi(t,x,z_a)$.\\
	Using the integration by parts for the transformation $y_t = \varphi(t,x,\bar{y}_t)$, it can be proved that there is a one-one corresponding between the solution of 
	\begin{equation}\label{RDEgbounded}
	dy_t= [Ay_t+f(y_t)]dt + g(y_t)dx_t = F(y_t)dt + g(y_t)dx_t.
	\end{equation}
	and the solution of the ordinary differential equation
	\begin{equation}\label{RDEtransformed}
	\dot{\bar{y}}_t = \Big[\frac{\partial \varphi}{\partial z}(t,x,\bar{y}_t)\Big]^{-1} F(\varphi(t,x,\bar{y}_t)).	
	\end{equation}
	Since the right hand side of \eqref{RDEtransformed} satisfies the global Lipschitz continuity and linear growth, by similar arguments as in \cite{riedelScheutzow} there exists a unique solution given the initial value. That in turn proves the existence and uniqueness of system \eqref{RDEgbounded}. A similar conclusion holds for the backward equation see e.g. \cite[Section 5.4]{frizhairer}.	\\
\end{proof}	
Thanks to the integrability of $\exp\{N_{\gamma,[a,b],\tp}(\bx)\}$, we can prove the integrability of the solution under the supremum norm $\|\cdot\|_\infty$ and the $\ltn \cdot,\cdot \rtn_{x,\tp}$ semi-norm. The reader is also referred to \cite{riedelScheutzow} for a similar version for the integrability of the solutions, defined in the sense of Friz-Victoir, of rough differential equation \eqref{RDE1}. Notice that the solutions of rough differential equation \eqref{RDE1} in the sense of Gubinelli and in the sense of Friz-Victoir could be proved to coincide. 
 
\begin{theorem}[{\bf Integrability of the solution}]\label{integrabilitysolution}
	For any interval $a<b\leq a+1$, the seminorm $\ltn y,y^\prime \rtn_{x,\tp,[a,b]}$ and the supremum norm $\|y\|_{\infty,[a,b]}$ are integrable.
\end{theorem}
\begin{proof}
Consider the solution mapping
	\begin{eqnarray*}
		&& \cM: \cD^{2\beta}_x(y_a,g(y_a)) \to \cD^{2\beta}_x(y_a,g(y_a)),\\
		&& \cM (y,y^\prime)_t := (H(y,y^\prime)_t, g(y_t)),\quad \text{where} \quad
		H(y,y^\prime)_t = y_a + \int_a^t F(y_s)ds+ \int_a^t g(y_s)dx_s. 
	\end{eqnarray*}
Given $\beta = \frac{1}{p}+\sigma, q = \frac{p}{2}$ and $(y,y^\prime) \in \cD_x^{2\beta}$, we would use the modified seminorm
\begin{eqnarray*}
	\ltn y,y^\prime \rtn_{x,\tp,I} &=& \ltn y^\prime \rtn_{p{\rm -var},I} +   \ltn R^y\rtn_{\tq,I},\qquad \text{where} \\
	\ltn R^y \rtn_{\tq,I} &:=& \Big(\sup_{\Pi} \sum_{[u,v] \in \Pi} \|R^y_{u,v}\|^q |v-u|^{-\sigma q}\Big)^{\frac{1}{q}} \geq |I|^{-\sigma} \ltn R^y \rtn_{q{\rm -var},I} \geq |I|^{-\sigma} \|R^y_{\min{I},\max{I}}\|.
\end{eqnarray*}
Observe that
	\begin{eqnarray}\label{Rg}
	&&g(y_t) - g(y_s) \notag\\
	&=& \int_0^1 D_yg(y_s + \eta y_{s,t})(y_s^\prime x_{s,t} + R^y_{s,t}) d\eta \notag \\
	&=& D_yg(y_s)y^\prime_s x_{s,t} + \int_0^1 \Big[D_yg(y_s + \eta y_{s,t}) - D_yg(y_s)\Big]y^\prime_s x_{s,t}d\eta + \int_0^1 D_yg(y_s + \eta y_{s,t})R^y_{s,t} d\eta\notag\\
	\end{eqnarray}
	hence $g(y)^\prime_s = D_yg(y_s)y^\prime_s$ where $y^\prime = g(y)$. Notice that 
	\begin{eqnarray*}
		\|g(y)^\prime\|_{\infty} &\leq& C_g \|y^\prime\|_\infty  \leq C_g (\|y^\prime_a\| + \ltn y^\prime \rtn_{p{\rm -var}}) \leq C_g^2 \|y_a\| + C_g \ltn y,y^\prime \rtn_{x,\tp} ;\\
		\ltn g(y) \rtn_{p{\rm -var}} &\leq& C_g \ltn y \rtn_{p{\rm -var}};\\
		\ltn g(y)^\prime \rtn_{p{\rm -var}} &\leq& \| D_y g(y)\|_\infty \ltn y^\prime \rtn_{p{\rm -var}} + \ltn D_yg(y) \rtn_{p{\rm -var}} \|y^\prime\|_\infty \leq C_g \ltn y, y^\prime \rtn_{x,\tp} + C_g^2 \ltn y \rtn_{p{\rm -var}};
	\end{eqnarray*}
	where 
	\begin{eqnarray}\label{zalpha}
	\ltn y \rtn_{p{\rm -var}} &\leq& \|y^\prime\|_\infty \ltn x \rtn_{p{\rm -var}} + (T-a)^{\sigma} \ltn R^y \rtn_{\tp} \notag\\
	&\leq& C_g \|y_a\|\ltn x \rtn_{p{\rm -var}} + \Big((T-a)^{\sigma} +\ltn x \rtn_{p{\rm -var}}\Big) \ltn y, y^\prime \rtn_{x,\tp} . 
	\end{eqnarray}
	On the other hand, it follows from \eqref{Rg} that
	\[
	\|R^{g(y)}_{s,t}\| \leq C_g \|R^y_{s,t}\| + \frac{1}{2} C_g  \|y^\prime\|_\infty \|x_{s,t}\| \|y_{s,t}\|,
	\]
	which, combined with \eqref{zalpha}, implies 
	\begin{eqnarray*}
		&&\ltn R^{g(y)} \rtn_{q{\rm -var}} \\
		&\leq& C_g \ltn R^y\rtn_{q{\rm -var}} + \frac{1}{2} C_g \|y^\prime\|_\infty \ltn x \rtn_{p{\rm -var}} \ltn y \rtn_{p{\rm -var}}   \\
		&\leq& C_g |t-s|^{\sigma} \ltn R^y \rtn_{\tq} + \frac{1}{2} C_g^2 \ltn x \rtn_{p{\rm -var}} \ltn y \rtn_{p{\rm -var}}\\
		&\leq& \frac{1}{2}C_g^3 \ltn x \rtn_{p{\rm -var}}^2 \|y_a\| + \Big(C_g |t-s|^{\sigma} +\frac{1}{2}C_g^2 \ltn x \rtn_{p{\rm -var}}  (\ltn x \rtn_{p{\rm -var}} +|t-s|^{\sigma})\Big)\ltn y,y^\prime \rtn_{x,\tp}. 
	\end{eqnarray*}
	Now we compute 
	\begin{eqnarray*}
		\|R^{H(y,y^\prime)}_{s,t}\| &=& (t-s)\Big\{ \|F(y_a)\|+L_f (C_g \|y_a\|+ \ltn y,y^\prime \rtn_{x,\tp}) \ltn x \rtn_{p{\rm -var}}+ L_f (T-a)^{\sigma} \ltn R^y \rtn_{\tp}\Big \} \\
		&&+ \Big\|\int_s^t [g(y_u)-g(y_s)]dx_u \Big\| \\
		&\leq& (t-s)\Big\{ L_f\|y_a\|+L_f (C_g \|y_a\|+ \ltn y,y^\prime \rtn_{x,\tp})\ltn x \rtn_{p{\rm -var}}+ L_f (T-a)^{\sigma} \ltn R^y \rtn_{\tp}\Big \}\\
		&&+\|g(y)^\prime_s\| \|\X_{s,t}\| + C_\alpha \Big(\ltn x \rtn_{p{\rm -var}} \ltn R^{g(y)}\rtn_{q{\rm -var}} + \ltn \X \rtn_{q{\rm -var}} \ltn g(y)^\prime \rtn_{p{\rm -var}}  \Big),  
	\end{eqnarray*}
	thus $\ltn R^{H(y,y^\prime)}\rtn_{\tq} $ can be estimated as follows
	\allowdisplaybreaks
	\begin{eqnarray*}
		&&	\ltn R^{H(y,y^\prime)}\rtn_{\tq} \\
		&\leq& |T-a|^{1-\sigma}\Big\{ L_f\|y_a\|+L_f (C_g \|y_a\|+ \ltn y,y^\prime \rtn_{x,\tp}) \ltn x \rtn_{p{\rm -var}}+ L_f (T-a)^{\sigma} \ltn y,y^\prime\rtn_{x,\tp}\Big \}\\
		&&+\Big(C_g^2 \|y_a\| + C_g \ltn y,y^\prime \rtn_{x,\tp} \Big) \ltn \X\rtn_{\tq} \\
		&& +  C_\alpha \ltn x \rtn_{\tp} \Big\{\frac{1}{2}C_g^3 \ltn x \rtn_{\tp} \|y_a\|+ \Big(C_g +\frac{1}{2}C_g^2 \ltn x \rtn_{\tp} +\frac{1}{2}C_g^2  \ltn x \rtn_{p{\rm -var}}\Big)\ltn y,y^\prime \rtn_{x,\tp}  \Big\}\\
		&& + C_\alpha \ltn \X \rtn_{\tp} \Big\{C_g \ltn y, y^\prime \rtn_{x,\tp} + C_g^3 \|y_a\|\ltn x \rtn_{p{\rm -var}} + C_g^2\Big((T-a)^{\sigma} +\ltn x \rtn_{p{\rm -var}}\Big) \ltn y, y^\prime \rtn_{x,\tp}\Big\}.
	\end{eqnarray*}
	In summary, we then get
	\begin{eqnarray}\label{zomeganorm}
	&&\ltn g(y) \rtn_{p{\rm -var}}+ \ltn R^{H(y,y^\prime)} \rtn_{\tq} \notag\\
	&\leq& \|y_a\| \Big\{C_g^2 \ltn x\rtn_{p{\rm -var}} + L_f (T-a)^{1-\sigma} + L_fC_g (T-a)^{1-\sigma} \ltn x\rtn_{p{\rm -var}} + C_g^2 \ltn \X \rtn_{\tq} \notag\\
	&&+ \frac{1}{2} C_\alpha C_g^3 \ltn x \rtn_{p{\rm -var}} \ltn x\rtn_{\tp} + C_\alpha C_g^3 \ltn x\rtn_{p{\rm -var}} \ltn \X \rtn_{\tq} \Big\} \notag\\
	&& + \Big\{C_g |T-a|^{\sigma} + C_g \ltn x \rtn_{p{\rm -var}} + L_f(T-a)^{1-\sigma} \ltn x \rtn_{p{\rm -var}} + L_f |T-a| + C_g \ltn \X\rtn_{\tq} \notag\\
	&& + C_\alpha \ltn x \rtn_{\tp} \Big(C_g +\frac{1}{2}C_g^2 \ltn x \rtn_{\tp} +\frac{1}{2}C_g^2  \ltn x \rtn_{p{\rm -var}}\Big) \notag\\
	&&+ C_\alpha \ltn \X\rtn_{\tq} \Big(C_g+ C_g^2 (T-a)^{\sigma} +C_g^2\ltn x \rtn_{p{\rm -var}}\Big) \Big\}  \ltn y, y^\prime \rtn_{x,2\alpha}.
	\end{eqnarray}
	Denote by
	\[
	M := \max \{ C_\alpha C_g (1+ C_g^2) , L_f (C_g+1), \frac{1}{3} \}
	\]
	the maximum of all the coefficients in the above estimates, then using the fact that
	\[
	\ltn x \rtn_{p{\rm -var}} \leq  |T-a|^{\sigma} \ltn x \rtn_{\tp} \leq \ltn x \rtn_{\tp},
	\]
	we derive from \eqref{zomeganorm} that
	\begin{eqnarray*}
	\ltn y,y^\prime \rtn_{x,\tp} &\leq& 3M\Big\{\ltn \X\rtn_{\tq} + \ltn x \rtn_{\tp} + (T-a)^{\sigma}\Big\}(\|y_a\|+\ltn y,y^\prime \rtn_{x,\tp}).
	\end{eqnarray*}
	Defining for any fixed $\mu \in (0,1)$ a sequence of greedy time $\{\bar{\tau}_i(\gamma,I,\alpha)\}_{i \in \N}$ as in \eqref{greedywitht}	then the estimate $\ltn y,y^\prime \rtn_{x,\tp}$ on each interval $[\bar{\tau}_i,\bar{\tau}_{i+1}]$ has the form
	\begin{eqnarray*}
		\ltn y,y^\prime \rtn_{x,\tp,[\bar{\tau}_i,\bar{\tau}_{i+1}]} \leq \mu \|y_{\bar{\tau}_i}\| + \mu \ltn y,y^\prime \rtn_{x,\tp,[\bar{\tau}_i,\bar{\tau}_{i+1}]} \Rightarrow\ltn y,y^\prime \rtn_{x,\tp,[\bar{\tau}_i,\bar{\tau}_{i+1}]} \leq \frac{\mu}{1-\mu}\|y_{\bar{\tau}_i}\|; 
	\end{eqnarray*}
	Therefore by applying Lemma \ref{additive}, we get
	\begin{eqnarray}\label{integ1}
	\ltn y,y^\prime \rtn_{x,\tp,[a,b]} &\leq& \bar{N}_{\frac{\mu}{3M},[a,b],\alpha}(\bx) \sum_{i=0}^{\bar{N}_{\frac{\mu}{3M},[a,b],\tp}(\bx)-1} \ltn y,y^\prime \rtn_{x,\tp,[\bar{\tau}_i,\bar{\tau}_{i+1}]}\notag\\
	&\leq & \bar{N}_{\frac{\mu}{3M},[a,b],\tp}(\bx) \sum_{i=0}^{\bar{N}_{\frac{\mu}{3M},[a,b],\tp}(\bx)-1}\frac{\mu}{1-\mu}\|y_{\bar{\tau}_i}\|.
	\end{eqnarray}
	To estimate $\|y_{\bar{\tau}_i}\|$ we use the fact that $y$ is controlled by $x$ to get
	\begin{eqnarray}\label{ytaubari}
	\|y_{\bar{\tau}_{i+1}}\| &\leq& \|y\|_{\infty,[\bar{\tau}_i,\bar{\tau}_{i+1}]} \leq \|y_{\bar{\tau}_i}\| +C_g \|y_{\bar{\tau}_i}\| \ltn x \rtn_{p{\rm -var},[\bar{\tau}_i,\bar{\tau}_{i+1}]}  + (\bar{\tau}_{i+1}-\bar{\tau}_i)^{\sigma}\ltn y,y^\prime \rtn_{x,\tp,[\bar{\tau}_i,\bar{\tau}_{i+1}]} \notag\\
	&\leq&\|y_{\bar{\tau}_i}\| (1 + \mu + \frac{\mu}{1-\mu})\leq\frac{1+\mu}{1-\mu}\|y_{\bar{\tau}_i}\|,
	\end{eqnarray}
	hence by induction
	\[
	\|y_{\bar{\tau}_{i}}\| \leq \Big(\frac{1+\mu}{1-\mu}\Big)^i \|y_a\|, \quad \forall i = 0,\dots, \bar{N}_{\frac{\mu}{3M},[a,b],\tp}(\bx).  
	\]
	We then conclude that
	\begin{eqnarray}\label{yomega}
	\ltn y,y^\prime \rtn_{x,\tp,[a,b]} &\leq& \bar{N}_{\frac{\mu}{3M},[a,b],\tp}(\bx) \sum_{i=0}^{\bar{N}_{\frac{\mu}{3M},[a,b],\tp}(\bx)-1}  \frac{\mu}{1-\mu} \Big(\frac{1+\mu}{1-\mu}\Big)^i \|y_a\| \notag\\
	&<& \frac{1}{2}\|y_a\| \bar{N}_{\frac{\mu}{3M},[a,b],\tp}(\bx)\exp\Big\{ \Big[\bar{N}_{\frac{\mu}{3M},[a,b],\tp}(\bx)\Big]\log\frac{1+\mu}{1-\mu}\Big\}.
	\end{eqnarray}
	Meanwhile the same estimate as \eqref{ytaubari} also shows that
	\begin{equation}\label{ysupest2}
	\|y\|_{\infty,[a,b]} \leq \|y_a\| \exp\Big\{ \Big[\bar{N}_{\frac{\mu}{3M},[a,b],\tp}(\bx)\Big]\log\frac{1+\mu}{1-\mu}\Big\}.
	\end{equation}
	Finally, the integrability of solution is a direct consequence of Corollary \ref{Nintegrability} on the integrability of $\exp\{\bar{N}_{\frac{\mu}{3M},[a,b],\tp}(\bx)\}$.
\end{proof}	

\section{Stability results}
We now formulate the main result of our paper.
\begin{theorem}[Asymptotic stability for rough differential equations]\label{globalRDE}
	Assume $\frac{1}{2}>\bar{\nu}> \nu >\frac{1}{3}$ and $X_\cdot(\omega)$ is a centered Gaussian process with stationary increments satisfying \eqref{Gaussianexpect}. Assume further that conditions \eqref{lambda}, \eqref{condf} are satisfied, where $g\in  C^3_b$ with coefficient $C_g$ and $\lambda > h(0)$. Then there exists an $\epsilon >0$ such that given $C_g < \epsilon$, the zero solution of \eqref{RDE1} is locally exponentially stable for almost all realization $x$ of $X$. If in addition $\lambda > C_f$, then we can choose $\epsilon$ so that the zero solution of \eqref{RDE1} is globally exponentially stable a.s.  
\end{theorem}

\begin{proof}
	The sketch of the proof is as follows. We derive the equation for $\log \|y_t\|$ in \eqref{logx1} and the equation for $\theta$ in \eqref{RDEangle0}. With the help of Proposition \ref{angleome} the estimate of $\ltn (\theta,\theta^\prime) \rtn_{x,2\alpha,[a,b]} $ is then given in \eqref{alphanormtheta}. Notice that for Gaussian geometric rough path, then $[x]_{\cdot,\cdot} = 0$, but we still compute the estimates here for general rough paths. {\bf Step 2} is to compute all components in \eqref{logx21}, in order to derive \eqref{logxn} and \eqref{logxn1} for $\log \|y_n\|$. The integrability of $\exp \Big \{ \bar{N}_{\frac{\mu}{2M},[k,k+1],\tp}(\bx)\Big\}$ then helps to choose $C_g < \epsilon$ small enough so that the arguments in \cite[Lemma 3.3]{duc19part1} can be applied to prove the local exponential stability. Finally, under the assumption $\lambda_A > C_f$, we derive \eqref{normy2} and \eqref{normy212} in {\bf Step 3}. The estimates for Young and rough integrals help to conclude that there exists an integrable $\kappa$ satisfying \eqref{normy218}, which follows the globally exponential stability for $C_g < \epsilon $ small enough.  \\
	
	{\bf Step 1.} We use similar arguments in \cite{duchongcong18} to prove that the solution of the pathwise solution of the linear rough differential equation \eqref{RDE1} generates a linear rough flow on $\R^d$, and that $y_t = 0$ iff $y_0=0$. Hence it remains to prove all the formula for $y_t$ and $r_t$. By direct computations using \eqref{Roughformula}, we can show the following equations. 
	\begin{itemize}
		\item $\|y_t\|^2$ satisfies the RDE
		\begin{equation*}
		d\|y_t\|^2 = 2 \langle y_t, Ay_t + f(y_t) \rangle dt + 2 \langle y_t, g(y_t) \rangle dx_t + \|g(y_t)\|^2 d[x]_{0,t},  
		\end{equation*}
		where $2\langle y,g(y) \rangle^\prime_s = 2 \langle y^\prime_s, g(y_s)\rangle + 2 \langle y_s, [g(y)]^\prime_s \rangle$.
		\item $\|y_t\|$ satisfies the RDE
		\begin{eqnarray*}
			d\|y_t\| &=& \frac{1}{\|y_t\|} \langle y_t, Ay_t + f(y_t)\rangle dt + \frac{1}{\|y_t\|} \langle y_t, g(y_t)\rangle dx_t \\
			&& + \frac{1}{2\|y_t\|}\Big[\|g(y_t)\|^2 -\frac{1}{\|y_t\|^2} \langle y_t, g(y_t) \rangle^2\Big]d[x]_{0,t},  
		\end{eqnarray*}
		where $\Big[\frac{1}{\|y\|}\langle y,g(y)\rangle\Big] ^\prime_s = \Big[ \frac{1}{\|y\|}\Big]^\prime_s \langle y_s,g(y_s)\rangle + \frac{1}{\|y_s\|} \Big[\langle y,g(y) \rangle\Big]^\prime_s$.
		\item $\log \|y_t\|$ satisfies the RDE
		\begin{eqnarray}\label{logx1}
		d\log \|y_t\| &=& \langle \theta_t, A\theta_t + \frac{f(y_t)}{\|y_t\|}\rangle dt +  \langle \theta_t, \frac{g(y_t)}{\|y_t\|}\rangle dx_t  + \Big[\frac{1}{2}\|\frac{g(y_t)}{\|y_t\|}\|^2 - \langle \theta_t, \frac{g(y_t)}{\|y_t\|} \rangle^2\Big]d[x]_{0,t}, \notag \\ 
		\end{eqnarray}
		where $\Big[\langle \theta, \frac{g(y)}{\|y\|} \rangle\Big]^\prime_s =  \langle \theta^\prime_s, \frac{g(y_s)}{\|y_s\|}\rangle + \langle \theta_s, [\frac{g(y)}{\|y\|}]^\prime_s \rangle$.
		\item $\theta_t$ satisfies the RDE
		\begin{eqnarray*}
			d\theta_t&=& \Big[A \theta_t -  \langle \theta_t, A\theta_t \rangle \theta_t + \frac{f(y_t)}{\|y_t\|} - \langle \theta_t,\frac{f(y_t)}{\|y_t\|} \rangle \theta_t \Big]dt + \Big[\frac{g(y_t)}{\|y_t\|}-  \langle \theta_t, \frac{g(y_t)}{\|y_t\|} \rangle \theta_t \Big] dx_t \notag\\
			&&+ \frac{1}{2}\Big\{ 3\langle \theta_t, \frac{g(y_t)}{\|y_t\|}\rangle^2 \theta_t- 2 \langle \theta_t, \frac{g(y_t)}{\|y_t\|}\rangle \frac{g(y_t)}{\|y_t\|}  - \|\frac{g(y_t)}{\|y_t\|}\|^2\theta_t  \Big\}d[x]_{0,t},
		\end{eqnarray*}
		where 
		\[
		\Big[\frac{g(y)}{\|y\|} - \langle \theta,\frac{g(y)}{\|y\|} \rangle \theta\Big]^\prime_s = [\frac{g(y)}{\|y\|}]^\prime_s - \langle \theta_s,\frac{g(y)}{\|y\|}\rangle \theta^\prime_s - \Big[\langle \theta_s,\frac{g(y_s)}{\|y_s\|}\rangle \Big]^\prime_s \theta_s.
		\]
	\end{itemize}
	
	Assign
	\[	
	G(y_s,\theta_s):=\frac{g(y_s)}{\|y_s\|} = \frac{g(y_s)-g(0)}{\|y_s\|} = \frac{\int_0^1 D_yg(\eta y_s) y_s d\eta}{\|y_s\|} = \int_0^1 D_y g(\eta y_s) \theta_s d\eta, 
	\]
	then it is easy to check that
	\[
	\|G(y,\theta)\|_\infty \leq C_g;\quad \ltn G(y,\theta) \rtn_{p{\rm - var}} \leq C_g(\ltn y\rtn_{p{\rm -var}} + \ltn \theta \rtn_{p{\rm -var}}). 
	\]
	Rewrite $\theta_t$ in the form
		\allowdisplaybreaks
	\begin{eqnarray}\label{RDEangle0}
	\theta_t&=& \theta_a + \int_a^t \Big[A \theta_s -  \langle \theta_s, A\theta_s \rangle \theta_s + \frac{f(y_s)}{\|y_s\|} - \langle \theta_s,\frac{f(y_s)}{\|y_s\|} \rangle \theta_s \Big]ds \notag \\
	&&+ \int_a^t \Big[G(y_s,\theta_s)-  \langle \theta_s, G(y_s,\theta_s) \rangle \theta_s \Big] dx_s \notag\\
	&&+ \int_a^t \frac{1}{2}\Big\{ 3\langle \theta_s, G(y_s,\theta_s)\rangle^2 \theta_s- 2 \langle \theta_s, G(y_s,\theta_s)\rangle G(y_s,\theta_s)  - \|G(y_s,\theta_s)\|^2\theta_s  \Big\}d[x]_{0,s}\notag\\
	&=& \theta_a + \int_a^t P(y_s,\theta_s)ds + \int_a^t Q(y_s,\theta_s)dy_s + \int_a^t H(y_s,\theta_s)d[x]_{0,s}.
	\end{eqnarray}
	We can prove an estimate for $\ltn (\theta,\theta^\prime) \rtn_{x,2\alpha,[a,b]}$. (The proof is provided in the Appendix).
	
	\begin{proposition}\label{angleome}
			For all $0 \leq a \leq b $, there exist a generic constant $P= P(b-a,\nu-\alpha)$ and a generic increasing function $Q(\cdot) = Q_{b-a,\nu-\alpha}(\cdot)$ such that 
		\begin{eqnarray}\label{alphanormtheta}
		&&\ltn (\theta,\theta^\prime) \rtn_{x,p,[a,b]} \notag\\
		&\leq& \frac{(b-a)^2(8M)^{\frac{2}{\nu}}}{4(1-\mu)\mu^{\frac{2}{\nu}-1}}\Big\{1+ \ltn y,y^\prime \rtn_{x,p,[a,b]}^\frac{2}{\nu}+\Big(\ltn x \rtn_{\nu,[a,b]} +  \ltn \X \rtn_{2\nu,\Delta^2([a,b])} +  \ltn [x]\rtn_{2\nu,\Delta^2([a,b])} \Big)^\frac{2}{\nu} \notag\\
		&&\qquad \qquad +\frac{1}{2}\ltn y,y^\prime \rtn_{x,p,[a,b]}^\frac{4}{\nu} + \frac{1}{2} \Big(\ltn x \rtn_{\nu,[a,b]} +  \ltn \X \rtn_{2\nu,\Delta^2([a,b])} +  \ltn [x]\rtn_{2\nu,\Delta^2([a,b])} \Big)^\frac{4}{\nu}\Big\},
		\end{eqnarray}
		where
		\begin{equation}\label{Mest}
		M := \max \Big\{2(C_f+\|A\|),96 K_\alpha (1 + C_\alpha) C_g^2(1+C_g),\frac{1}{2}\Big \} 
		\end{equation}
	\end{proposition}

	{\bf Step 2.} It is now sufficient to estimate the quantity in \eqref{logx1}. For any integer $n >1$, rewrite \eqref{logx1} in the integral form
	\begin{eqnarray}\label{logx21}
	\log \|y_t\| 
	&=& \log \|y_a\| + \int_a^t \langle \theta_s, A\theta_s + \frac{f(y_s)}{\|y_s\|}\rangle ds + \int_a^t \langle \theta_s, G(y_s,\theta_s) \rangle dx_s \notag \\
	&& +\int_a^t \Big[\frac{1}{2}\|G(y_s,\theta_s)\|^2 - \langle \theta_s, G(y_s,\theta_s) \rangle^2\Big]d[x]_{0,s}\notag\\
	&\leq& \log \|y_a\| -\lambda_A (t-a) +\int_a^t h(\|y_s\|)ds + \Big\|\int_a^t \langle \theta_s, G(y_s,\theta_s) \rangle dx_s\Big\| \notag\\
	&& + \Big\|\int_a^t \Big[\frac{1}{2}\|G(y_s,\theta_s)\|^2 - \langle \theta_s, G(y_s,\theta_s) \rangle^2\Big]d[x]_{0,s}\Big\|. 
	\end{eqnarray}
	The Young integral in the last line of \eqref{logx21} can be estimated as
	\allowdisplaybreaks
	\begin{eqnarray}\label{logxyoung1}
	&&\Big\|\int_a^t \Big[\frac{1}{2}\|G(y_s,\theta_s)\|^2 - \langle \theta_s, G(y_s,\theta_s)\rangle^2\Big]d[x]_{a,s} \Big\| \notag \\
	&\leq& \frac{3}{2}C_g^2 \Big\|[x]_{a,t}\Big\| + K_\alpha \ltn [x] \rtn_{\frac{p}{2}{\rm -var},\Delta^2([a,t])} \ltn\Big[\frac{1}{2}\|G(y,\theta)\|^2 - \langle \theta, G(y,\theta)\rangle^2\Big] \rtn_{p{\rm - var},[a,t]}\notag\\
	&\leq& \frac{3}{2} C_g^2 \ltn [x] \rtn_{\frac{p}{2}{\rm -var},\Delta^2([a,t])} \notag\\
	&&+ K_\alpha \ltn [x] \rtn_{\frac{p}{2}{\rm -var},\Delta^2([a,t])} \Big[C_g^2(\ltn y\rtn_{p{\rm -var}} +\ltn \theta\rtn_{p{\rm -var}}) + 2 C_g^2(\ltn y\rtn_{p{\rm -var}} +2\ltn \theta\rtn_{p {\rm -var}}) \Big] \notag\\
	& \leq&  \frac{3}{2} C_g^2 \ltn [x] \rtn_{\frac{p}{2}{\rm -var},\Delta^2([a,t])}  + 3K_\alpha C_g^2 \ltn [x] \rtn_{\frac{p}{2}{\rm -var},\Delta^2([a,t])}  \Big(C_g \ltn x \rtn_{p{\rm - var}} + \ltn y,y^\prime \rtn_{x,p}\Big)\notag \\
	&&+ 5 K_\alpha C_g^2 \ltn [x] \rtn_{\frac{p}{2}{\rm -var},\Delta^2([a,t])}  \Big( 2 C_g \ltn x \rtn_{p{\rm - var}} + \ltn \theta,\theta^\prime \rtn_{x,p} \Big)\notag\\
	&\leq& \frac{3}{2} C_g^2 \ltn [x] \rtn_{p{\rm -var},\Delta^2([a,t])} + 13K_\alpha C_g^3 \ltn [x] \rtn_{\frac{p}{2}{\rm -var},\Delta^2([a,t])} \ltn x \rtn_{p{\rm -var}}\notag\\
	&&+K_\alpha C_g^2 \ltn [x] \rtn_{\frac{p}{2}{\rm -var},\Delta^2([a,t])}\Big( 3 \ltn y,y^\prime \rtn_{x,p} + 5 \ltn \theta,\theta^\prime \rtn_{x,p}\Big)\notag\\
	&\leq& \frac{3}{2} C_g^2 \ltn [x] \rtn_{p{\rm -var},\Delta^2([a,t])} + 13K_\alpha C_g^3 \ltn [x] \rtn_{\frac{p}{2}{\rm -var},\Delta^2([a,t])} \ltn x \rtn_{p{\rm -var}}\notag\\
	&& + 17K_\alpha^2 C_g^2 \ltn [x] \rtn_{\frac{p}{2}{\rm -var},\Delta^2([a,t])}^2 + \frac{1}{2} C_g^2 \ltn y,y^\prime \rtn_{x,p}^2 + \frac{1}{2} C_g^2 \ltn \theta,\theta^\prime \rtn_{x,p}^2. 
	\end{eqnarray}
	Meanwhile the rough integral can be estimated as 
	\begin{eqnarray}\label{logxrough0}
	&& \Big\|\int_a^t \langle \theta_s, G(y_s,\theta_s) \rangle dx_s\Big\| \notag\\
	&&\leq \Big\|\langle \theta_a,  G(y_a,\theta_a) \rangle\Big\| \|x_t-x_a\| + \Big\|\langle \theta,G(y,\theta) \rangle^\prime_a \Big\| \|\X_{a,t}\|\notag\\
	&& \qquad + C_\alpha \Big(\ltn x \rtn_{p{\rm -var},[a,t]} \ltn R^{\langle \theta,G(y,\theta) \rangle} \rtn_{\frac{p}{2}{\rm -var},[a,t]} + \ltn \langle  \theta,G(y,\theta) \rangle^\prime \rtn_{p{\rm -var},[a,t]} \ltn \X \rtn_{\frac{p}{2}{\rm -var},\Delta^2([a,t])} \Big)\notag\\
	&& \leq C_g \ltn x \rtn_{p{\rm -var},[a,t]} + 5 C_g^2 \ltn \X \rtn_{\frac{p}{2}{\rm -var},\Delta^2([a,t])} \notag\\
	&& \qquad + C_\alpha  \Big(\ltn x \rtn_{p{\rm -var},[a,t]} \ltn R^{\langle \theta,G(y,\theta) \rangle} \rtn_{\frac{p}{2}{\rm -var},[a,t]} + \ltn \langle  \theta,G(y,\theta) \rangle^\prime \rtn_{p{\rm -var},[a,t]} \ltn \X \rtn_{\frac{p}{2}{\rm -var},\Delta^2([a,t])} \Big).\notag\\
	\end{eqnarray}
	To estimate the brackets of the last line of \eqref{logxrough0}, we apply \eqref{yalpha1} and \eqref{xalpha1} to get
	\begin{eqnarray*}
		&&\ltn \langle  \theta,G(y,\theta)\rangle^\prime \rtn_{p{\rm -var}} \\
		&\leq & \ltn \langle \theta^\prime, G(y,\theta) \rangle \rtn_{p{\rm -var}} + \ltn \langle \theta, \frac{\partial G}{\partial y}(y,\theta) y^\prime \rangle \rtn_{p{\rm -var}} + \ltn \langle \theta, \frac{\partial G}{\partial \theta}(y,\theta) \theta^\prime \rangle \rtn_{p{\rm -var}}\\
		&\leq & C_g \ltn \theta^\prime \rtn_{p{\rm -var}} + C_g(\ltn y \rtn_{p{\rm -var}} +  \ltn \theta \rtn_{p{\rm -var}})\|\theta^\prime\|_\infty+ C_g \|y^\prime\|_\infty \ltn \theta \rtn_{p{\rm -var}} \\
		&&+  C_g \Big(\ltn y \rtn_{p{\rm -var}} \|y^\prime\|_\infty + \ltn \theta \rtn_{p{\rm -var}} \|y^\prime\|_\infty + \ltn y^\prime \rtn_{p{\rm -var}}\Big) \\
		&& + C_g \|y^\prime\|_\infty \ltn \theta \rtn_{p{\rm -var}} +  C_g \Big(\ltn y\rtn_{p{\rm -var}} \|\theta^\prime\|_\infty + \ltn \theta^\prime\rtn_{p{\rm -var}} \Big) \\
		&\leq& 10 C_g^2 \Big(C_g \ltn x \rtn_{p{\rm -var}} + \ltn y,y^\prime \rtn_{x,p}\Big) +13 C_g^2 \Big( 2 C_g \ltn x \rtn_{p{\rm -var}} + \ltn \theta,\theta^\prime \rtn_{x,p} \Big)
	\end{eqnarray*}
	which, together with Cauchy inequality follows that
	\begin{eqnarray}\label{yG}
	\ltn \langle  \theta,G(y,\theta)\rangle^\prime \rtn_{p{\rm -var}} \ltn \X \rtn_{\frac{p}{2}{\rm -var},\Delta^2([a,t])}	&\leq& 36 C_g^3  \ltn x \rtn_{p{\rm -var}}  \ltn \X \rtn_{\frac{p}{2}{\rm -var},\Delta^2([a,t])}+ \frac{23}{2} C_g^2\ltn \X \rtn_{\frac{p}{2}{\rm -var},\Delta^2([a,t])}^2\notag \\
	&& + 5C_g^2\ltn y,y^\prime \rtn_{x,p}^2 + \frac{13}{2}C_g^2\ltn \theta,\theta^\prime \rtn_{x,p}^2.    
	\end{eqnarray}
	In addition,
	\begin{eqnarray*}
		&&\ltn R^{\langle \theta,G(y,\theta) \rangle}\rtn_{\frac{p}{2}{\rm -var}}\\
		&\leq& \ltn \theta \rtn_{p{\rm -var}} \ltn G(y,\theta) \rtn_{p{\rm -var}} + \|G(y,\theta)\|_\infty \ltn R^\theta \rtn_{\frac{p}{2}{\rm -var}} + \ltn R^{G(y,\theta)} \rtn_{\frac{p}{2}{\rm -var}} \\ 
		&\leq& C_g\ltn \theta \rtn_{p{\rm -var}} (\ltn y \rtn_{p{\rm -var}} +\ltn \theta \rtn_{p{\rm -var}}) + C_g \ltn R^\theta \rtn_{\frac{p}{2}{\rm -var}} \\
		&&+ C_g \Big(\ltn R^y \rtn_{\frac{p}{2}{\rm -var}} + \|y^\prime\|_\infty \ltn x \rtn_{p{\rm -var}} \ltn y\rtn_{p{\rm -var}} + \ltn R^\theta \rtn_{\frac{p}{2}{\rm -var}} + \|\theta^\prime\|_\infty \ltn x \rtn_{p{\rm -var}} \ltn \theta\rtn_{p{\rm -var}} \Big)\\
		&\leq& C_g \ltn y,y^\prime \rtn_{x,p} + 2 C_g \ltn \theta,\theta^\prime \rtn_{x,p} + C_g^2 \ltn x \rtn_{p{\rm -var}} \Big(C_g \ltn x \rtn_{p{\rm -var}} + \ltn y,y^\prime \rtn_{x,p}\Big) \\
		&& + 2C_g^2 \ltn x \rtn_{p{\rm -var}} \Big(2C_g \ltn x \rtn_{p{\rm -var}} + \ltn \theta,\theta^\prime \rtn_{x,p}\Big) \\
		&& + C_g \Big(2C_g \ltn x \rtn_{p{\rm -var}} + \ltn \theta,\theta^\prime \rtn_{x,p}\Big)  \Big\{3C_g \ltn x \rtn_{p{\rm -var}} + \Big(\ltn y,y^\prime \rtn_{x,p}+\ltn \theta,\theta^\prime \rtn_{x,p}\Big)\Big\} 
	\end{eqnarray*}
	which together with Cauchy inequality gives
	\allowdisplaybreaks
	\begin{eqnarray} \label{RG}
	&& \ltn x \rtn_{p{\rm -var}} \ltn R^{\langle \theta,G(y,\theta) \rangle}\rtn_{\frac{p}{2}{\rm -var}}\notag \\
	&\leq& \frac{5}{2} C_g \ltn x \rtn_{p{\rm -var}}^2  + C_g\ltn y,y^\prime \rtn_{x,p}^2 +C_g\ltn \theta,\theta^\prime \rtn_{x,p}^2 + 5 C_g^3 \ltn x \rtn_{p{\rm -var}}^3 + 6 C_g^3 \ltn x \rtn_{p{\rm -var}}^3 + \frac{5}{2} C_g^2 \ltn x \rtn_{p{\rm -var}}^4 \notag \\
	&& + \frac{29}{2} C_g^2 \ltn x \rtn_{p{\rm -var}}^2 + \frac{1}{2}C_g^2\ltn y,y^\prime \rtn_{x,p}^2 + \frac{1}{2}C_g^2 \ltn \theta,\theta^\prime \rtn_{x,p}^2 +  \frac{1}{4}C_g\ltn y,y^\prime \rtn_{x,p}^4 + \frac{1}{2}C_g\ltn \theta,\theta^\prime \rtn_{x,p}^4 \notag\\
	&\leq& 	\frac{5}{2} C_g \ltn x \rtn_{p{\rm -var}}^2 + 11 C_g^3 \ltn x \rtn_{p{\rm -var}}^3 + 17 C_g^2 \ltn x \rtn_{p{\rm -var}}^4   \notag \\
	&&+ (C_g + C_g^2)\Big(\ltn y,y^\prime \rtn_{x,p}^2 + \ltn \theta,\theta^\prime \rtn_{x,p}^2 + \ltn y,y^\prime \rtn_{x,p}^4 + \ltn \theta,\theta^\prime \rtn_{x,p}^4\Big). 
	\end{eqnarray}
	Combining \eqref{logxyoung1}, \eqref{logxrough0},\eqref{yG}, \eqref{RG} to \eqref{logx21}, we conclude that there exists a generic constant $\Lambda$ and a generic polynomial $\kappa_1(\ltn x \rtn_\alpha, \ltn \X \rtn_{2\alpha}, \ltn [x] \rtn_{2\alpha})$ such that
	\begin{eqnarray*}
		\log \|y_t\| &\leq& \log \|y_a\| + \int_a^t [-\lambda_A +h(\|y_s\|)]ds \\
		&&+ C_g\kappa_1(\ltn x \rtn_{p{\rm -var},[a,t]}, \ltn \X \rtn_{\frac{p}{2}{\rm -var},\Delta^2([a,t])}, \ltn [x] \rtn_{\frac{p}{2}{\rm -var},\Delta^2([a,t])}) \notag \\
		&& + C_g \Lambda\Big(\ltn \theta,\theta^\prime \rtn_{x,p,[a,t]}^2 + \ltn y,y^\prime \rtn_{x,p,[a,t]}^2 + \ltn \theta,\theta^\prime \rtn_{x,p,[a,t]}^4 + \ltn y,y^\prime \rtn_{x,p,[a,t]}^4 \Big).
	\end{eqnarray*}
	Using \eqref{alphanormtheta} and Cauchy inquality, and with the generic constant $\Lambda$ and the generic polynomial $\kappa_1$ if necessary (that is possible since $\alpha <\nu$), we conclude that
	\begin{eqnarray}\label{logxn}
	&&\log \|y_t\| \notag\\
	&\leq& \log \|y_a\|+ \int_a^t \Big[-\lambda_A +h(\|y_s\|)+ C_g \Lambda\Big)\Big]ds  \notag\\
	&&+ C_g \kappa_1(\ltn x \rtn_{\nu,[a,t]}, \ltn \X \rtn_{2\nu,\Delta^2([a,t])}, \ltn [x] \rtn_{p{\rm -var},\Delta^2([a,t])}) + C_g \Lambda \Big\{ \ltn y,y^\prime \rtn_{x,p,[a,t]}^4 + \ltn y,y^\prime \rtn_{x,p,[a,t]}^\frac{16}{\nu}\Big\}.\notag\\
	\end{eqnarray}
	To estimate $\ltn y,y^\prime \rtn_{x,p,[a,t]}$, we apply \eqref{yomega} with generic constant $\Lambda$ to conclude that for any $a<t\leq b \leq a+1$,
	\begin{eqnarray}\label{xxp}
	\ltn y,y^\prime \rtn_{x,p,[a,t]}^m \leq \ltn y,y^\prime \rtn_{x,\tp,[a,t]}^m&\leq& \Lambda \|y_a\|^m \bar{N}^m_{\frac{\mu}{3M},[a,b],\tp}(\bx) \exp \{m\Lambda \bar{N}_{\frac{\mu}{2M},[a,t],\tp}(\bx)\} \notag\\
	&\leq& \frac{1}{2}\Lambda \|y_a\|^{2m} + \frac{1}{2}\Lambda \exp \Big \{ 2m\Lambda \bar{N}_{\frac{\mu}{2M},[a,t],\tp}(\bx) \Big\}.
	\end{eqnarray}
	By replacing \eqref{xxp} into \eqref{logxn}, there exists a generic polynomials with all positive coefficients 
	\[
	P\Big(\exp \Big \{ \bar{N}_{\frac{\mu}{2M},[a,t],\tp}(\bx)\Big\}\Big)
	\] 
	such that
	\begin{eqnarray}\label{logxn1}
\log \|y_t\| 	&\leq& \log \|y_a\|+\int_a^t \Big[-\lambda_A +h(\|y_s\|)+ C_g \Lambda\Big]ds \notag\\
	&&+ C_g \kappa_1\Big(\ltn x \rtn_{p{\rm -var},[a,t]}, \ltn \X \rtn_{\frac{p}{2}{\rm -var},\Delta^2([a,t])}, \ltn [x] \rtn_{\frac{p}{2}{\rm -var},\Delta^2([a,t])} \Big) \notag\\
	&&+ C_g P\Big(\exp \Big \{ \bar{N}_{\frac{\mu}{2M},[a,t],\tp}(\bx)\Big\}\Big) + C_g \kappa_2(\|y_a\|),
	\end{eqnarray}
	where 
	\[
	\kappa_2(z) = \frac{1}{2} \Lambda \Big(z^8 +  z^{\frac{32}{\nu}}\Big),
	\]
	for some generic constant $\Lambda$. Using \eqref{expect} and \eqref{Nbarergodic}, there exists for almost sure all $x$ the limit
	\begin{eqnarray*}
		&& \lim \limits_{n \to \infty} \frac{1}{n}\sum_{k = 0}^{n-1} \kappa_1(\ltn x \rtn_{p{\rm -var},[k,k+1]}, \ltn \X \rtn_{\frac{p}{2}{\rm -var},\Delta^2([k,k+1])}, \ltn [x] \rtn_{\frac{p}{2}{\rm -var},\Delta^2([k,k+1])}) \\
		&&= E \kappa_1(\ltn x \rtn_{p{\rm -var},[0,1]}, \ltn \X \rtn_{\frac{p}{2}{\rm -var},\Delta^2([0,1])}, \ltn [x] \rtn_{\frac{p}{2}{\rm -var},\Delta^2([0,1])}) = \kappa_1 < \infty;\\
		\text{and} &&\lim \limits_{n \to \infty} \frac{1}{n}\sum_{k = 0}^{n-1}P\Big(\exp \Big \{\Lambda \bar{N}_{\frac{\mu}{2M},[k,k+1],\tp}(\bx)\Big\}\Big)  = E P \Big(\exp \Big \{\Lambda \bar{N}_{\frac{\mu}{2M},[0,1],\tp}(\bx)\Big\}\Big) = \bar{P}< \infty,
	\end{eqnarray*}
	we can use \eqref{ysupest2} and the same arguments in \cite[Lemma 3.3 \& Lemma 3.4]{duc19part1} to conclude that the zero solution of \eqref{RDE1} is locally exponentially stable.\\
	
	{\bf Step 3.} Assume that $\lambda_A > C_f$ and assign $\lambda := \lambda_A - C_f$, then  $e^{2\lambda t}\|y_t\|^2$ satisfies the RDE
	\begin{equation}\label{normy2}
	d e^{2\lambda t}\|y_t\|^2 = 2e^{2\lambda t}\Big(\lambda \|y_t\|^2+\langle y_t, Ay_t + f(y_t) \rangle\Big) dt + 2 e^{2\lambda t}\langle y_t, g(y_t) \rangle dx_t + e^{2\lambda t}\|g(y_t)\|^2 d[x]_{0,t},  
	\end{equation}
	where 
	\begin{eqnarray*}
\Big[e^{2\lambda \cdot}\langle y,g(y) \rangle\Big]^\prime_s &=& e^{2\lambda s} \Big[ \langle y^\prime_s, g(y_s)\rangle + \langle y_s, [g(y)]^\prime_s \rangle\Big] + 2\lambda e^{2\lambda s}\langle y_s,g(y_s) \rangle\\
&=& e^{2\lambda s} \Big[ \|g(y_s)\|^2 + \langle y_s, D_g(y_s) g(y_s)\rangle\Big] + 2\lambda e^{2\lambda s}\langle y_s,g(y_s) \rangle.
	\end{eqnarray*}
	 Rewrite in the integral form 
	\begin{eqnarray}\label{normy21}
	e^{2 \lambda t} \|y_t\|^2 &=& \|y_0\|^2 + 2 \int_0^t e^{2 \lambda s} \Big(\lambda \|y_s\|^2 + \langle y_s, Ay_s+ f(y_s) \rangle\Big)ds \notag\\
	&&+ 2 \int_0^t e^{2 \lambda s} \langle y_s, g(y_s) \rangle dx_s + 2\int_0^t  e^{2\lambda s}\|g(y_s)\|^2 d[x]_{0,s}.
	\end{eqnarray}
	Using \eqref{lambda}, the first integral in \eqref{normy21} is then non-positive, thus for any $n \in \N$
	\begin{eqnarray}\label{normy212}
	e^{2 \lambda n} \|y_n\|^2 &\leq& \|y_0\|^2 + \sum_{k=0}^{n-1} 2\Big\|\int_k^{k+1} e^{2 \lambda s} \langle y_s, g(y_s) \rangle dx_s\Big\| + \sum_{k=0}^{n-1} 2\Big\|\int_k^{k+1} e^{2 \lambda s} \|g(y_s)\|^2 d[x]_{0,s}\Big\|.\notag\\
	\end{eqnarray}
	The Young integral in \eqref{normy212} can be estimated as
	\begin{eqnarray*}
	&&\Big\|\int_s^t e^{2\lambda u} \|g(y_u)\|^2 d[x]_{0,u} \Big\| \\
	&\leq& \ltn [x] \rtn_{q{\rm -var},[s,t]} \Big(C_g^2 e^{2\lambda s} \|y_s\|^2 + K \ltn e^{2\lambda \cdot} \|g(y)\|^2 \rtn_{p{\rm -var},[s,t]}\Big) \\
	&\leq& \ltn [x] \rtn_{q{\rm -var},[s,t]} \Big\{C_g^2 e^{2\lambda s} \|y_s\|^2 + K \Big(\ltn e^{2\lambda \cdot} \rtn_{p{\rm -var}} \|g(y)\|_{\infty,[s,t]}^2 + 2 \|e^{2\lambda \cdot}\|_{\infty,[s,t]} \|g(y)\|_\infty \ltn g(y)\rtn_{p{\rm -var},[s,t]}\Big)\Big\}\\
	&\leq& \ltn [x] \rtn_{q{\rm -var},[s,t]} \Big\{C_g^2 e^{2\lambda s} \|y_s\|^2+ K \Big( C_g^2 (e^{2\lambda t}-e^{2\lambda s})\|y\|^2_{\infty,[s,t]} + 2 C_g^2 e^{2\lambda t} \|y\|_{\infty,[s,t]}\ltn y\rtn_{p{\rm -var},[s,t]}\Big)\Big\}.
	\end{eqnarray*}
Since 
\begin{eqnarray*}
\ltn y\rtn_{p{\rm -var},[s,t]} &\leq& \|y^\prime\|_{\infty,[s,t]} \ltn x \rtn_{p{\rm -var},[s,t]} + (t-s)^\sigma \ltn y, y^\prime \rtn_{x,\tp,[s,t]}\\
&\leq& C_g \ltn x \rtn_{p{\rm -var},[s,t]} \|y\|_{\infty,[s,t]} + (t-s)^\sigma \ltn y, y^\prime \rtn_{x,\tp,[s,t]}, 
\end{eqnarray*}
by using \eqref{yomega} and \eqref{ysupest2}, we conclude that there exists a function $\kappa_1(t-s,\ltn x \rtn_{p{\rm -var},[s,t]}, \ltn [x] \rtn_{q{\rm -var},[s,t]})$ with 
\begin{equation}\label{normy215}
E \kappa_1(1,\ltn x \rtn_{p{\rm -var},[0,1]}, \ltn [x] \rtn_{q{\rm -var},[0,1]}) < \infty
\end{equation}
such that
\begin{equation}\label{normy213}
\Big\|\int_s^t e^{2\lambda u} \|g(y_u)\|^2 d[x]_{0,u} \Big\| \leq C_g \kappa_1(t-s,\ltn x \rtn_{p{\rm -var},[s,t]}, \ltn [x] \rtn_{q{\rm -var},[s,t]}) e^{2\lambda s} \|y_s\|^2.
\end{equation}
Meanwhile the rough integral in \eqref{normy212} can be estimated as
\begin{eqnarray}\label{normy214}
	&&\Big\|\int_s^t e^{2\lambda u} \langle y_u,g(y_u)\rangle dx_u \Big\| \notag\\
	&\leq& e^{2\lambda s} C_g \|y_s\| \ltn x \rtn_{p{\rm -var},[s,t]} + \left\|\Big[e^{2\lambda \cdot}\langle y,g(y) \rangle\Big]^\prime_s\right\| \ltn \X \rtn_{q{\rm -var},[s,t]} \notag\\
	&& + C_\alpha \Big(\ltn x \rtn_{p{\rm -var},[s,t]} \ltn R^{e^{2\lambda \cdot} \langle y,g(y)\rangle } \rtn_{q{\rm -var},[s,t]} + \ltn \Big[e^{2\lambda \cdot}\langle y,g(y) \rangle\Big]^\prime \rtn_{p{\rm -var},[s,t]} \ltn \X \rtn_{q{\rm -var},[s,t]} \Big). \notag\\
\end{eqnarray}
Hence by using estimates \eqref{yomega} and \eqref{ysupest2} and similar technique in {\bf Step 2} for estimating the right hand side of \eqref{normy214} we conclude that there exists a function $\kappa_2(t-s,\ltn x \rtn_{p{\rm -var},[s,t]}, \ltn \X \rtn_{q{\rm -var},[s,t]}, \ltn [x] \rtn_{q{\rm -var},[s,t]})$ with 
\begin{equation}\label{normy216}
E \kappa_2(1,\ltn x \rtn_{p{\rm -var},[0,1]}, \ltn \X \rtn_{q{\rm -var},[0,1]},\ltn [x] \rtn_{q{\rm -var},[0,1]}) < \infty
\end{equation}
such that
\begin{equation}\label{normy217}
\Big\|\int_s^t e^{2\lambda u} \langle y_u,g(y_u)\rangle dx_u \Big\| \leq C_g \kappa_2(t-s,\ltn x \rtn_{p{\rm -var},[s,t]},\ltn \X \rtn_{q{\rm -var},[s,t]}, \ltn [x] \rtn_{q{\rm -var},[s,t]}) e^{2\lambda s} \|y_s\|^2.
\end{equation}
By replacing \eqref{normy213} and \eqref{normy217} into \eqref{normy212}, we conclude that there exists an integrable function $\kappa(\dots) = \kappa_1(\dots)+ \kappa_s(\dots)$ such that
\begin{equation}\label{normy218}
e^{2 \lambda n} \|y_n\|^2 \leq \|y_0\|^2 + \sum_{k=0}^{n-1} 2C_g \kappa (1,\ltn x \rtn_{p{\rm -var},[k,k+1]},\ltn \X \rtn_{q{\rm -var},[k,k+1]}, \ltn [x] \rtn_{q{\rm -var},[k,k+1]}) e^{2\lambda k} \|y_k\|^2.
\end{equation}
Similar to the arguments in {\bf Step 3} of \cite[Theorem 3.5]{duc19part1}, we  apply the discrete Gronwall lemma in \cite[Lemma 4]{ducGANSch18} to conclude that 
\[
\limsup \limits_{t \to \infty} \frac{1}{t} \log \|y_t\| \leq -\lambda + \frac{1}{2}E \log \Big[1+2C_g \kappa \Big(1,\ltn x \rtn_{p{\rm -var},[0,1]},\ltn \X \rtn_{q{\rm -var},[0,1]}, \ltn [x] \rtn_{q{\rm -var},[0,1]}\Big)\Big].
\]
Hence there exists a $\epsilon$ small enough such that for any $C_g < \epsilon$, the zero solution is globally exponentially stable a.s. We note that unlike the local stability, the integrability of functions $\kappa$ is not necessary, but only the integrability of $\log (1+ C_g \kappa(\dots))$.

\end{proof}


\section{Appendix}

\subsection{Some technical lemmas}

\begin{lemma}\label{additive}
	Let $x\in \widehat{C}^{p}([a,b],\R^d)$, $p\geq 1$. If $a = \tau_0<\tau_1<\cdots < \tau_N = b$, then 
	\begin{eqnarray*}
	\ltn x\rtn_{p\text{-}\rm{var},[a,b]}&\leq& N^{\frac{p-1}{p}}\sum_{i=0}^{N-1}\ltn x\rtn_{p{\rm -var},[\tau_i,\tau_{i+1}]};\\
	\ltn x\rtn_{\tp,[a,b]}&\leq& N^{\frac{p-1}{p}}\sum_{i=0}^{N-1}\ltn x\rtn_{\tp,[\tau_i,\tau_{i+1}]}.
	\end{eqnarray*}
\end{lemma}	
\begin{proof}
	The first estimate is a direct consequence of \cite[Lemma 2.1]{congduchong17}. The second estimate is followed from the inequality that
	\[
	(a_1 + \ldots + a_k)^p (b_1 + \ldots + b_k)^{-\sigma p} \leq k^{p-1} \Big( a_1^p b_1^{-\sigma p} + \ldots +  a_k^p b_k^{-\sigma p} \Big),\quad \forall a_i, b_i > 0, i = 1,\dots,k. 
	\]
\end{proof}	

\begin{lemma}\label{roughintegral}
	Given $\beta = \frac{1}{p}+\sigma$, , assume that $(y,y^\prime) \in \cD_x^{2\beta}(I)$ and $\Gamma_{s,t} := \int_s^t y_u dx_u$. Then $(\Gamma,\Gamma^\prime) \in \cD^{2\beta}_x(I)$ with $\Gamma^\prime = y$. Moreover for any $I$ such that $|I| \leq 1$
	\begin{equation}\label{gammaest}
	\ltn \Gamma, \Gamma^\prime \rtn_{x,\tp,I} \leq \Big(\|y^\prime_{\min{I}}\| + (C_{p,|I|}+1)\ltn y,y^\prime \rtn_{x,\tp,I}\Big)\Big(|I|^\sigma + \ltn x \rtn_{p{\rm -var},I} + \ltn \X \rtn_{\tq,I}\Big). 
	\end{equation}
\end{lemma}
\begin{proof}
	It follows directly from \eqref{roughEst} that $(\Gamma,\Gamma^\prime) \in \cD^{2\beta}_x(I,\R^d)$ with $\Gamma^\prime_s = y_s$. As a result,
	\begin{eqnarray}\label{gammap}
	\ltn \Gamma^\prime \rtn_{p{\rm - var},I} = \ltn y \rtn_{p{\rm -var},I} &\leq& \|y^\prime\|_{\infty,I} \ltn x \rtn_{p{\rm -var},I} + \ltn R^y \rtn_{q {\rm -var},I}  \notag\\
	&\leq& (\|y^\prime_{\min{I}}\| + \ltn y^\prime \rtn_{p{\rm -var},I}) \ltn x \rtn_{p{\rm - var},I} + |I|^\sigma \ltn R^y \rtn_{\tq,I}.
	\end{eqnarray}
	On the other hand,
	\begin{eqnarray*}
	\| R^\Gamma_{s,t} \| = \Big \|\int_s^t y_{s,u} dx_u \Big\| \leq \|y^\prime\|_{\infty,I} \|\X_{s,t}\| + C_{p,|I|} |t-s|^\sigma\Big(\ltn x \rtn_{p{\rm -var}} \ltn R^y \rtn_{\tq} + \ltn y^\prime \rtn_{p{\rm -var}} \ltn \X \rtn_{\tq} \Big),   
	\end{eqnarray*}
which implies 
\begin{equation}\label{Rgamma}
\ltn R^\Gamma \rtn_{\tq,I} \leq (\|y^\prime_{\min{I}}\| + \ltn y^\prime \rtn_{p{\rm -var},I}) \ltn \X \rtn_{\tq} + C_{p,|I|} (\ltn x \rtn_{p{\rm -var}} + \ltn \X \rtn_{\tq}) \ltn y,y^\prime \rtn_{x,\tp,I}.
\end{equation}
Combining \eqref{gammap} and \eqref{Rgamma} and using $|I| \leq 1$ we get
\begin{eqnarray*}
&&\ltn \Gamma, \Gamma^\prime \rtn_{x,\tp,I} \\
&\leq&  (\|y^\prime_{\min{I}}\| + \ltn y^\prime \rtn_{p{\rm -var},I}) (\ltn x\rtn_{p{\rm -var},I} + \ltn \X \rtn_{\tq,I}) + |I|^\sigma \ltn y,y^\prime \rtn_{x,\tp,I} \notag\\
&&+ C_{p,|I|} (\ltn x \rtn_{p{\rm -var}} + \ltn \X \rtn_{\tq}) \ltn y,y^\prime \rtn_{x,\tp,I}\notag\\
&\leq& \|y^\prime_{\min{I}}\|(\ltn x\rtn_{p{\rm -var},I} + \ltn \X \rtn_{\tq,I}) + \Big(|I|^\sigma + \ltn x \rtn_{p{\rm -var},I} + \ltn \X \rtn_{\tq,I}\Big) (C_{p,|I|} + 1) \ltn y,y^\prime \rtn_{x,\tp,I} 
\end{eqnarray*}
which implies \eqref{gammaest}.
\end{proof}

\begin{lemma}\label{roughfunction}
Assume that $g \in C^3_b$ with coefficient $C_g$ and $z_u = g(y_u)$. Then $(z,z^\prime) \in \cD^{2\beta}_x(I)$ with $z^\prime_s = Dg(y_s)y^\prime_s$ and for any $|I|\leq 1$ we get
\begin{equation}\label{zroughest}
\ltn z, z^\prime \rtn_{x,\tp,I} \leq C_g\Big[ 1+ \|y^\prime\|_{\infty,I}\Big(|I|^\sigma + \frac{3}{2} \ltn x\rtn_{\tp,I} + \frac{1}{2} \ltn x \rtn_{\tp,I}^2 \Big) \Big] \Big(\|y^\prime_{\min{I}}\| + \ltn y, y^\prime \rtn_{x,\tp,I}\Big). 
\end{equation} 
\end{lemma}	
\begin{proof}
	Since
	\begin{eqnarray*}
	g(y_t) - g(y_s) &=& \int_0^1 Dg(y_s + \eta y_{s,t})y_{s,t}d\eta =  \int_0^1 Dg(y_s + \eta y_{s,t})(y^\prime_s x_{s,t} + R^y_{s,t})d\eta\\
	&=& Dg(y_s)y^\prime_s x_{s,t} + \int_0^1 Dg(y_s + \eta y_{s,t})  R^y_{s,t} d\eta + \int_0^1 [Dg(y_s + \eta y_{s,t}) - Dg(y_s)] y^\prime_s x_{s,t} d\eta,
	\end{eqnarray*}
it follows that $z^\prime_s = Dg(y_s)y^\prime_s$ and
\begin{eqnarray*}
\|R^z_{s,t} \| \leq C_g \|R^y_{s,t}\| + \int_0^1 C_g \eta \|y_{s,t}\| \|y^\prime\|_{\infty,I} \|x_{s,t}\| d\eta,
\end{eqnarray*}
so that
\[
\ltn R^z \rtn_{\tq,I} \leq C_g \ltn R^y \rtn_{\tq,I} + \frac{1}{2}C_g \|y^\prime\|_{\infty,I} \ltn x\rtn_{\tp,I} \ltn y \rtn_{p{\rm -var},I},
\]
where 
\begin{eqnarray*}
\ltn y \rtn_{p{\rm -var},I} &\leq& \|y^\prime\|_{\infty,I}\ltn x \rtn_{p{\rm -var},I} + |I|^\sigma \ltn R^y \rtn_{\tq,I} \\
&\leq& \Big(\|y^\prime_{\min{I}}\| +  \ltn y^\prime \rtn_{p{\rm -var},I} \Big) |I|^\sigma \ltn x \rtn_{\tp,I} + |I|^\sigma \ltn R^y \rtn_{\tq,I}\\
&\leq& (|I|^\sigma +\ltn x \rtn_{\tp,I}) \Big( \|y^\prime_{\min{I}}\| + \ltn y,y^\prime \rtn_{x,\tp,I}\Big).
\end{eqnarray*}
On the other hand,
\[
\ltn z^\prime \rtn_{p{\rm -var},I} \leq \ltn Dg(y) y^\prime \rtn_{p{\rm -var},I} \leq C_g \ltn y^\prime \rtn_{p{\rm -var},I} + C_g \| y^\prime\|_{\infty,I} \ltn y \rtn_{p{\rm -var},I}.
\]
Hence given $|I| \leq 1$ we get
\allowdisplaybreaks
\begin{eqnarray*}
&&\ltn z,z^\prime \rtn_{x,2\alpha,I} \\
&\leq& C_g \ltn y^\prime \rtn_{p{\rm -var},I} + C_g \| y^\prime\|_{\infty,I} \ltn y \rtn_{p{\rm -var},I} + C_g \ltn R^y \rtn_{\tq,I} + \frac{1}{2}C_g \|y^\prime\|_{\infty,I} \ltn x\rtn_{\tp,I} \ltn y \rtn_{p{\rm -var},I}\notag\\
&\leq& C_g \ltn y,y^\prime \rtn_{x,\tp,I} + C_g\|y^\prime\|_{\infty,I}\Big(1 + \frac{1}{2}\ltn x \rtn_{\tp,I}\Big)(|I|^\sigma +\ltn x \rtn_{\tp,I}) \Big( \|y^\prime_{\min{I}}\| + \ltn y,y^\prime \rtn_{x,\tp,I}\Big)
\end{eqnarray*}
which follows \eqref{zroughest} due to $|I|\leq 1$.
\end{proof}	
	
\begin{lemma}\label{roughfunctiondiff}
	Assume that $g \in C^3_b$ with coefficient $C_g$ and $z_u = g(\bar{y}_u)-g(y_u)$. Then $(z,z^\prime) \in \cD^{2\beta}_x(I)$ with $z^\prime_s = Dg(\bar{y}_s)\bar{y}^\prime_s - Dg(y_s) y^\prime_s$. In addition for any $|I|\leq 1$ we get
	\begin{eqnarray}\label{zroughdiffest}
&&	\ltn z, z^\prime \rtn_{x,\tp,I}\notag\\
&\leq& 4C_g(1 + \ltn x \rtn_{\tp,I} + |I|^\sigma)^2\Big[ \|\bar{y}_{\min{I}}\|+\|\bar{y}^\prime_{\min{I}}\| + \|y_{\min{I}}\| +\|y^\prime_{\min{I}}\| + \ltn \bar{y},\bar{y}^\prime \rtn_{x,\tp,I} + \ltn y,y^\prime \rtn_{x,\tp,I}\Big] \times \notag \\
	&&\times \Big(\|\bar{y}_{\min{I}}-y_{\min{I}}\|+\|\bar{y}^\prime_{\min{I}}-y^\prime_{\min{I}}\| + \ltn \bar{y}-y, \bar{y}^\prime - y^\prime \rtn_{x,\tp,I}\Big).
	\end{eqnarray}
\end{lemma}	

\begin{proof}
	A direct computation shows that
		\allowdisplaybreaks
	\begin{eqnarray*}
	z_{s,t} &=& g(\bar{y}_t) - g(y_t) - g(\bar{y}_s) + g(y_s) \\
	&=& \int_0^1 \Big[ Dg(\bar{y}_s + \eta \bar{y}_{s,t}) \bar{y}_{s,t} - Dg(y_s + \eta y_{s,t})y_{s,t} \Big] d\eta \\
	&=& \int_0^1 \Big[ Dg(\bar{y}_s + \eta \bar{y}_{s,t}) R^{\bar{y}}_{s,t} - Dg(y_s + \eta y_{s,t})R^y_{s,t} \Big] d\eta \\
	&&+ \int_0^1 \Big[ Dg(\bar{y}_s + \eta \bar{y}_{s,t}) \bar{y}^\prime_s - Dg(y_s + \eta y_{s,t})y^\prime_s \Big] x_{s,t}d\eta \\
	&=& \int_0^1 \Big[ Dg(\bar{y}_s + \eta \bar{y}_{s,t}) R^{\bar{y}}_{s,t} - Dg(y_s + \eta y_{s,t})R^y_{s,t} \Big] d\eta + [Dg(\bar{y}_s)\bar{y}^\prime_s - Dg(y_s) y^\prime_s] x_{s,t}\\
	&&+ \int_0^1 \int_0^1 \Big[ D^2g(\bar{y}_s + \eta \eta_1 \bar{y}_{s,t})[\bar{y}_{s,t},\bar{y}^\prime_{s}] - D^2g (y_s + \eta \eta_1 y_{s,t})[y_{s,t},y^\prime_s]\Big] \eta d \eta_1 x_{s,t}d\eta.
	\end{eqnarray*}
This shows that $z^\prime_s = Dg(\bar{y}_s)\bar{y}^\prime_s - Dg(y_s) y^\prime_s$ and 
\begin{eqnarray*}
z^\prime_{s,t} &=& Dg(\bar{y}_t)\bar{y}_t^\prime - Dg(y_t)y_t^\prime - Dg(\bar{y}_s)\bar{y}_s^\prime + Dg(y_s)y_s^\prime \\
&=& [Dg(\bar{y}_t)-Dg(\bar{y}_s)]\bar{y}^\prime_t - [Dg(y_t)-Dg(y_s)]y^\prime_t + [Dg(\bar{y}_s)-Dg(y_s)]\bar{y}^\prime_{s,t}- Dg(y_s)[\bar{y}_{s,t}-y^\prime_{s,t}]\\
&=& \int_0^1 \Big[D^2g(\bar{y}_s + \eta \bar{y}_{s,t})[\bar{y}_{s,t},\bar{y}^\prime_t] - D^2g(y_s + \eta y_{s,t})[y_{s,t},y^\prime_t] \Big]d\eta\\
&&+ [Dg(\bar{y}_s)-Dg(y_s)]\bar{y}^\prime_{s,t}- Dg(y_s)[\bar{y}_{s,t}-y^\prime_{s,t}];
\end{eqnarray*}
which implies that
\begin{eqnarray}
\ltn z^\prime \rtn_{p{\rm -var}} &\leq& C_g \Big[\|\bar{y}-y\|_{\infty} + \|\bar{y}_{\cdot,\cdot} - y_{\cdot,\cdot}\|_{\infty} \Big] \ltn \bar{y} \rtn_{p{\rm -var}} \|\bar{y}^\prime\|_\infty + C_g \ltn \bar{y} - y \rtn_{p{\rm -var}} \|\bar{y}^\prime\|_\infty \notag\\
&&+ C_g \ltn y \rtn_{p{\rm -var}} \|\bar{y}^\prime - y^\prime\|_\infty + C_g \|\bar{y}-y\|_\infty \ltn \bar{y}^\prime \rtn_{p{\rm -var}} + C_g \ltn \bar{y}^\prime - y^\prime \rtn_{p{\rm -var}}.
\end{eqnarray}
A similar estimate shows that
\begin{eqnarray*}
\ltn R^z \rtn_{\tq} &\leq& C_g \ltn R^{\bar{y}} \rtn_{\tq} \Big[\|\bar{y}-y\|_\infty + \|\bar{y}_{\cdot,\cdot}-y_{\cdot,\cdot}\|_\infty \Big] + C_g \ltn R^{\bar{y}} - R^y \rtn_{\tq} \\
&&+ \ltn x \rtn_{\tp} \Big\{C_g \|\bar{y}^\prime\|_\infty \ltn \bar{y} \rtn_{p{\rm -var}} \Big[\|\bar{y}-y\|_\infty + \|\bar{y}_{\cdot,\cdot} - y_{\cdot,\cdot}\|_\infty\Big] \\
&& \qquad \qquad \qquad + C_g \|\bar{y}^\prime\|_\infty \ltn \bar{y} - y\rtn_{p{\rm -var}} + C_g \ltn y \rtn_{p{\rm -var}} \|\bar{y}^\prime - y^\prime \|_\infty\Big\}.
\end{eqnarray*}
Therefore 
\begin{eqnarray}\label{zzprime}
&&\ltn z,z^\prime \rtn_{x,\tp}\\
&\leq& \Big\{\|\bar{y}-y\|_\infty + \|\bar{y}_{\cdot,\cdot} - y_{\cdot,\cdot}\|_\infty + \ltn \bar{y} - y \rtn_{p{\rm -var}} + \|\bar{y}^\prime - y^\prime\|_\infty + \ltn \bar{y}^\prime - y^\prime \rtn_{p{\rm -var}} + \ltn R^{\bar{y}} - R^y \rtn_{\tq}\Big\} \times \notag\\
&\times& \Big\{(C_g  \|\bar{y}^\prime\|_\infty \ltn \bar{y} \rtn_{p{\rm -var}} + C_g \ltn y \rtn_{p{\rm -var}})(1+\ltn x \rtn_{\tp}) + C_g +C_g  \|\bar{y}^\prime\|_\infty (1+ \ltn x \rtn_{\tp}) + C_g \ltn \bar{y},\bar{y}^\prime \rtn_{x,\tp} \Big \}.\notag
\end{eqnarray}
Note that
\begin{eqnarray*}
\|\bar{y}-y\|_\infty &\leq& \|\bar{y}_{\min{I}} -y_{\min{I}}\| + \|\bar{y}^\prime_{\min{I}} -y^\prime_{\min{I}}\| |I|^\sigma \ltn x\rtn_{\tp}  + (|I|^\sigma + \ltn x \rtn_{\tp}) \ltn \bar{y}-y,\bar{y}^\prime - y^\prime \rtn_{x,\tp}; \\
\|\bar{y}_{\cdot,\cdot} - y_{\cdot,\cdot}\|_\infty &\leq& \|\bar{y}^\prime_{\min{I}} -y^\prime_{\min{I}}\| |I|^\sigma \ltn x \rtn_{\tp} + (|I|^\sigma + \ltn x \rtn_{\tp}) \ltn \bar{y}-y,\bar{y}^\prime - y^\prime \rtn_{x,\tp};\\
\ltn \bar{y} - y\rtn_{p{\rm -var}} &\leq& |I|^\sigma \ltn x \rtn_{\tp} \|\bar{y}^\prime_{\min{I}} -y^\prime_{\min{I}}\|  + (|I|^\sigma + \ltn x \rtn_{\tp}) \ltn \bar{y}-y,\bar{y}^\prime - y^\prime \rtn_{x,\tp};\\
\|\bar{y}^\prime - y^\prime\|_\infty &\leq& \|\bar{y}^\prime_{\min{I}} -y^\prime_{\min{I}}\|  + \ltn \bar{y}-y,\bar{y}^\prime - y^\prime \rtn_{x,\tp};\\
\ltn \bar{y} \rtn_{p{\rm -var}} &\leq & \|\bar{y}_{\min{I}}\| +  \|\bar{y}^\prime_{\min{I}}\| |I|^\sigma \ltn x \rtn_{\tp}+ (|I|^\sigma + \ltn x\rtn_{\tp}) \ltn \bar{y},\bar{y}^\prime \rtn_{x,\tp}; \\
\ltn y \rtn_{p{\rm -var}} &\leq & \|y_{\min{I}}\| + \|y^\prime_{\min{I}}\| |I|^\sigma \ltn x\rtn_{\tp}+ (|I|^\sigma + \ltn x\rtn_{\tp}) \ltn y,y^\prime \rtn_{x,\tp}.
\end{eqnarray*}
Replacing all the above estimates into \eqref{zzprime}, we get \eqref{zroughdiffest}.
\end{proof}		

\subsection{Proofs of auxilliary propositions}

\begin{proof}[{\bf Proposition \ref{angleome}}]
	 We are going to estimate the H\"older norm of $\theta$ using equation \eqref{RDEangle0}. Consider the solution mapping $\cM: \cD^{2\alpha}_x(\theta_a,Q(x_a,\theta_a)) \to \cD^{2\alpha}_x(\theta_a,Q(x_a,\theta_a)) $ defined by
	\[
	\cM (\theta,\theta^\prime)_t = (F(\theta,\theta^\prime)_t, Q(y_t,\theta_t)),
	\]
	where $F$ is defined as the right hand side of \eqref{RDEangle0}, together with the seminorm
	\begin{eqnarray*}
		\ltn (\theta,\theta^\prime)\rtn_{x,p} = \ltn \theta^\prime \rtn_{p{\rm -var}} + \ltn R^\theta \rtn_{\frac{p}{2}{\rm -var}},\quad
		\ltn \cM (\theta,\theta^\prime)\rtn_{x,2\alpha} = \ltn Q(y,\theta) \rtn_{p{\rm -var}} + \ltn R^{F(\theta,\theta^\prime)} \rtn_{\frac{p}{2}{\rm -var}}.
	\end{eqnarray*}
	We are going to estimate these seminorms. From $\theta^\prime = Q(y,\theta)$ and $y^\prime = g(y)$, it follows that
	\allowdisplaybreaks
	\begin{eqnarray}
	\ltn \theta\rtn_{p{\rm -var}} &\leq&\ltn \theta^\prime \rtn_\infty \ltn x \rtn_{p{\rm -var}} + \ltn R^\theta\rtn_{\frac{p}{2}{\rm -var}}\leq 2C_g \ltn x \rtn_{p{\rm -var}} + \ltn \theta,\theta^\prime \rtn_{x,p};\label{yalpha1}\\	
	\ltn y\rtn_{p{\rm -var}} &\leq&\ltn y^\prime \rtn_\infty \ltn x \rtn_{p{\rm -var}} + \ltn R^y\rtn_{\frac{p}{2}{\rm -var}} \leq C_g \ltn x \rtn_{p{\rm -var}} + \ltn y,y^\prime \rtn_{x,p};\label{xalpha1}\\	
	\ltn Q(y,\theta) \rtn_{p{\rm -var}} &\leq& 2 \ltn G(y,\theta) \rtn_{p{\rm -var}} + 2 C_g \ltn \theta \rtn_{p{\rm -var}} \leq 4 C_g \ltn \theta \rtn_{p{\rm -var}} + 2 C_g \ltn y \rtn_{p{\rm -var}}; \label{gyalpha1}\\
	\|Q(y,\theta)\|_\infty &\leq& 2 \|G(y,\theta)\|_\infty \leq 2 C_g;\notag \\
	\|H(y,\theta)\|_\infty &\leq& 3 \|G(y,\theta)\|_\infty^2 \leq 12 C_g^2; \notag \\
	\ltn H(y,\theta)\rtn_{p{\rm -var}} &\leq& 6 \|G(y,\theta)\|_\infty^2 \ltn \theta \rtn_{p{\rm -var}} + 6 \|G(y,\theta)\|_\infty \ltn G(y,\theta)\rtn_{p{\rm -var}} \leq 6C_g^2\Big( 2\ltn \theta \rtn_{p{\rm -var}} +   \ltn y \rtn_{p{\rm -var}}\Big).\notag 
	\end{eqnarray}
	Meanwhile
	\begin{eqnarray*}
		[Q(y,\theta)]^\prime_s &=& \frac{\partial Q}{\partial y}(y_s,\theta_s)y^\prime_s + \frac{\partial Q}{\partial \theta}(y_s,\theta_s)\theta^\prime_s\\
		&=& \frac{\partial G}{\partial y}(y_s,\theta_s)y^\prime_s - \langle \theta_s, \frac{\partial G}{\partial y}(y_s,\theta_s)y^\prime_s \rangle \theta_s + \frac{\partial G}{\partial \theta}(y_s,\theta_s)\theta^\prime_s\\
		&&- \langle \theta^\prime_s, G(y_s,\theta_s) \rangle \theta_s - \langle \theta_s, G(y_s,\theta_s) \rangle \theta^\prime_s - \langle \theta_s, \frac{\partial G}{\partial \theta}(y_s,\theta_s)\theta^\prime_s \rangle \theta_s,
	\end{eqnarray*}
	where 
	\[
	\frac{\partial G}{\partial y} = \int_0^1 D_{yy}g(\eta y) \theta \eta d\eta, \quad  \frac{\partial G}{\partial \theta} = \int_0^1 D_yg(\eta y) d\eta;
	\]
	which, together with $\theta^\prime = Q(y,\theta)$ and $y^\prime = g(y)$, derive
	\allowdisplaybreaks
	\begin{eqnarray*}
		&&\ltn Q(y,\theta)^\prime \rtn_{p{\rm -var}} \\
		&\leq& C_g \Big(\ltn y \rtn_{p{\rm -var}} \|y^\prime\|_\infty + \ltn \theta \rtn_{p{\rm -var}} \|y^\prime\|_\infty + \ltn y^\prime \rtn_{p{\rm -var}}\Big) \\
		&&+ C_g \Big( 2\|y^\prime\|_\infty \ltn \theta \rtn_{p{\rm -var}} + \ltn y \rtn_{p{\rm -var}} \|y^\prime\|_\infty + \ltn \theta \rtn_{p{\rm -var}} \|y^\prime\|_\infty + \ltn y^\prime \rtn_{p{\rm -var}}\Big)\\
		&& + 2C_g\Big( \ltn \theta^\prime \rtn_{p{\rm -var}} + \|\theta^\prime\|_\infty \ltn \theta \rtn_{p{\rm -var}} + 2 \|\theta^\prime\|_\infty \ltn \theta\rtn_{p{\rm -var}} + 2 \|\theta^\prime\|_\infty \ltn y\rtn_{p{\rm -var}}\Big)\\
		&&+ C_g \Big(\ltn y\rtn_{p{\rm -var}} \|\theta^\prime\|_\infty + \ltn \theta^\prime\rtn_{p{\rm -var}} \Big) + C_g\Big(2 \|\theta^\prime\|_\infty \ltn \theta \rtn_{p{\rm -var}} + \ltn \theta^\prime \rtn_{p{\rm -var}} + \|\theta^\prime\|_\infty \ltn y \rtn_{p{\rm -var}} \Big)\\
		&\leq& 2C_g \Big( \ltn y \rtn_{p{\rm -var}} \|y^\prime\|_\infty +  \ltn y^\prime \rtn_{p{\rm -var}} + 2 \ltn \theta \rtn_{p{\rm -var}} \|y^\prime\|_\infty + 2 \ltn \theta^\prime\rtn_{p{\rm -var}} \\
		&&+ 3 \ltn y\rtn_{p{\rm -var}} \|\theta^\prime\|_\infty +  4 \|\theta^\prime\|_\infty \ltn \theta \rtn_{p{\rm -var}} \Big)\\
		&\leq& 2C_g \Big( C_g\ltn y \rtn_{p{\rm -var}} +  \ltn g(y) \rtn_{p{\rm -var}} + 2C_g \ltn \theta \rtn_{p{\rm -var}}   + 2 \ltn Q(y,\theta)\rtn_{p{\rm -var}} \\
		&& + 6 C_g \ltn y\rtn_{p{\rm -var}} +  8C_g \ltn \theta \rtn_{p{\rm -var}} \Big)\\
		&\leq& 12 C_g^2 \Big(2 \ltn y\rtn_{p{\rm -var}} + 3 \ltn \theta \rtn_{p{\rm -var}} \Big)
	\end{eqnarray*}
	and 
	\[
	\| Q(y,\theta)^\prime \|_\infty \leq 2C_g \Big(\|y^\prime\|_\infty +2\|\theta^\prime\|_\infty \Big) \leq 2 C_g \Big(C_g + 4C_g \Big) \leq 10 C_g^2.
	\]
	Hence by using H\"older inequality, we get
	\begin{eqnarray}\label{RF}
	&&\|R^{F(\theta,\theta^\prime)}_{s,t}\| \notag\\
	& \leq& \int_s^t \|P(y_u,\theta_u)\|du  + \|[Q(y,\theta)]^\prime_s\| \|\X_{s,t}\| +\|H(y,\theta)\|_\infty \|[x]_{s,t}\| + K_\alpha \ltn H(y,\theta) \rtn_{p{\rm -var}} \ltn [x] \rtn_{\frac{p}{2}{\rm -var}}\notag\\
	&&+ C_\alpha \Big(\ltn x\rtn_{p{\rm -var}} \ltn R^{Q(y,\theta)}\rtn_{\frac{p}{2}{\rm -var}} + \ltn Q(y,\theta)^\prime \rtn_{p{\rm -var}} \ltn \X \rtn_{\frac{p}{2}{\rm -var}}\Big) \notag\\
	&\leq& 2(C_f + \|A\|)(t-s) + 10 C_g^2 \|\X_{s,t}\|+ 12C_g^2 \|[x]_{s,t}\| \notag\\
	&& + 6K_\alpha C_g^2 \ltn [x] \rtn_{\frac{p}{2}{\rm -var}} \Big(2\ltn \theta \rtn_{p{\rm -var}} +   \ltn y \rtn_{p{\rm -var}} \Big) \notag\\
	&&+ C_\alpha \Big(\ltn x\rtn_{p{\rm -var}} \ltn R^{Q(y,\theta)}\rtn_{\frac{p}{2}{\rm -var}} + 12 C_g^2 (2 \ltn y\rtn_{p{\rm -var}} + 3 \ltn \theta \rtn_{p{\rm -var}} ) \ltn \X \rtn_{\frac{p}{2}{\rm -var}}\Big),
	\end{eqnarray}
	On the other hand
	\begin{eqnarray*}
		\|R^{Q(y,\theta)}_{s,t}\| &\leq& \left\|Q(y_t,\theta_t) - Q(y_s,\theta_t) - \frac{\partial Q}{\partial y} (y_s,\theta_s) y^\prime_s x_{s,t} \right\| \\
		&& +\left\| Q(y_s,\theta_t) - Q(y_s,\theta_s) - \frac{\partial Q}{\partial \theta} (y_s,\theta_s) \theta^\prime_s x_{s,t}\right \|\\
		&\leq&\int_0^1 \Big\|\frac{\partial Q}{\partial y}\Big(y_s + \eta (y_t-y_s),\theta_t\Big) -  \frac{\partial Q}{\partial y}(y_s,\theta_t)\Big\| \|y^\prime_s\| \|x_{s,t}\|d\eta  \\
		&& + \int_0^1 \Big\|\frac{\partial Q}{\partial \theta}\Big(x_s,\theta_s + \eta (\theta_t-\theta_s)\Big) -  \frac{\partial Q}{\partial x}(x_s,\theta_s)\Big\| \|\theta^\prime_s\| \|x_{s,t}\|d\eta \\
		&& + \int_0^1 \Big\|\frac{\partial Q}{\partial y}\Big(y_s + \eta (y_t-y_s),\theta_t\Big) \Big\| d \eta\ \|R^y_{s,t}\| + \int_0^1 \Big\|\frac{\partial Q}{\partial \theta}\Big(y_s,\theta_s + \eta (\theta_t-\theta_s)\Big) \Big\| d \eta\ \|R^\theta_{s,t}\|, 
	\end{eqnarray*}
	thus
	\begin{eqnarray*}
		&&\ltn R^{Q(y,\theta)}\rtn_{\frac{p}{2}{\rm -var}} \\
		&\leq&  C_g \ltn R^y\rtn_{\frac{p}{2}{\rm -var}} + C_g \|y^\prime\|_\infty \ltn x \rtn_{p{\rm -var}}  \ltn y\rtn_{p{\rm -var}}  + C_g \ltn R^\theta\rtn_{\frac{p}{2}{\rm -var}} + C_g \|\theta^\prime\|_\infty \ltn x \rtn_{p{\rm -var}}  \ltn \theta\rtn_{p{\rm -var}}  \\
		&\leq& C_g \Big( \ltn R^y\rtn_{\frac{p}{2}{\rm -var}}  +  \ltn R^\theta\rtn_{\frac{p}{2}{\rm -var}}\Big) + C_g^2 \ltn y \rtn_{p{\rm -var}}  \ltn x\rtn_{p{\rm -var}}  +  2C_g^2 \ltn x \rtn_{p{\rm -var}}  \ltn \theta\rtn_{p{\rm -var}}  \\
		&\leq& C_g \Big( \ltn R^y\rtn_{\frac{p}{2}{\rm -var}}  +  \ltn R^\theta\rtn_{\frac{p}{2}{\rm -var}}\Big)  + C_g^2 \ltn x \rtn_{p{\rm -var}}  \Big(\ltn y\rtn_{p{\rm -var}}  + 2 \ltn \theta\rtn_{p{\rm -var}} \Big).
	\end{eqnarray*}
	Combining these above estimates into \eqref{RF} and using \eqref{gyalpha1}, we get
	\allowdisplaybreaks
	\begin{eqnarray*}
		&&\ltn R^{F(\theta,\theta^\prime)}\rtn_{\frac{p}{2}{\rm -var}} + \ltn Q(y,\theta) \rtn_{p{\rm -var}} \\
		&\leq& 2(C_f + \|A\|)(T-a)+10 C_g^2 \ltn \X\rtn_{\frac{p}{2}{\rm -var}}+ 12C_g^2 \ltn [x]\rtn_{\frac{p}{2}{\rm -var}} \notag\\
		&& + 6K_\alpha C_g^2 \ltn [x] \rtn_{\frac{p}{2}{\rm -var}}\Big(2\ltn \theta \rtn_{p{\rm -var}} +   \ltn y \rtn_{p{\rm -var}} \Big) + 2C_g \Big(\ltn y \rtn_{p{\rm -var}} + 2\ltn \theta\rtn_{p{\rm -var}} \Big) \notag\\
		&&+ C_\alpha \Big\{C_g \ltn x\rtn_{p{\rm -var}}  \Big( \ltn R^y\rtn_{\frac{p}{2}{\rm -var}}  +  \ltn R^\theta\rtn_{\frac{p}{2}{\rm -var}}\Big)  + C_g^2 \ltn x \rtn_{p{\rm -var}}^2  \Big(\ltn y\rtn_{p{\rm -var}} + 2 \ltn \theta\rtn_{p{\rm -var}}\Big) \\
		&&\qquad \qquad \qquad \qquad + 12 C_g^2 \Big(2 \ltn y\rtn_{p{\rm -var}} + 3 \ltn \theta \rtn_{p{\rm -var}} \Big) \ltn \X \rtn_{\frac{p}{2}{\rm -var}}\Big\}\\
		&\leq& 2(C_f + \|A\|)(T-a)+10 C_g^2 \ltn \X\rtn_{\frac{p}{2}{\rm -var}}+ 12C_g^2 \ltn [x]\rtn_{\frac{p}{2}{\rm -var}} \notag\\
		&&+ \Big(2 C_g + 6 K_\alpha C_g^2 \ltn [x] \rtn_{\frac{p}{2}{\rm -var}} + C_\alpha C_g^2 \ltn x \rtn_{p{\rm -var}}^2 + 24 C_\alpha C_g^2  \ltn \X\rtn_{\frac{p}{2}{\rm -var}} \Big) \ltn y \rtn_{p{\rm -var}} \\
		&&+ \Big(4 C_g + 12 K_\alpha C_g^2  \ltn [x] \rtn_{\frac{p}{2}{\rm -var}} + 2 C_\alpha C_g^2  \ltn x \rtn_{p{\rm -var}}^2 + 36 C_\alpha C_g^2  \ltn \X\rtn_{\frac{p}{2}{\rm -var}} \Big) \ltn \theta \rtn_{p{\rm -var}} \\
		&& + C_\alpha C_g  \ltn y,y^\prime \rtn_{x,p} + C_\alpha C_g  \ltn \theta,\theta^\prime \rtn_{x,p}. 
	\end{eqnarray*}
	Using \eqref{yalpha1} and \eqref{xalpha1}, it follows that for any $a < T$ such that $T-a \leq1$ we get
	\begin{eqnarray*}
		&&\ltn R^{F(\theta,\theta^\prime)}\rtn_{\frac{p}{2}{\rm -var}} + \ltn G(y,\theta) \rtn_{p{\rm -var}} \\
		&\leq&2(C_f + \|A\|)(T-a)+10 C_g^2 \ltn \X\rtn_{\frac{p}{2}{\rm -var}}+ 12C_g^2 \ltn [x]\rtn_{\frac{p}{2}{\rm -var}} \notag\\
		&&+ \Big(2 C_g + 6 K_\alpha C_g^2 \ltn [x] \rtn_{\frac{p}{2}{\rm -var}} + C_\alpha C_g^2 \ltn x \rtn_{p{\rm -var}}^2 + 24 C_\alpha C_g^2 \ltn \X\rtn_{\frac{p}{2}{\rm -var}} \Big) \Big(C_g \ltn x \rtn_{p{\rm -var}} + \ltn y,y^\prime\rtn_{x,p} \Big) \\
		&&+ \Big(4 C_g + 12 K_\alpha C_g^2 \ltn [x] \rtn_{\frac{p}{2}{\rm -var}} + 2 C_\alpha C_g^2 \ltn x \rtn_{p{\rm -var}}^2 + 36 C_\alpha C_g^2 \ltn \X\rtn_{\frac{p}{2}{\rm -var}} \Big)\Big(2C_g \ltn x \rtn_{p{\rm -var}} + \ltn \theta,\theta^\prime\rtn_{x,p} \Big) \\
		&& + C_\alpha C_g \ltn y,y^\prime \rtn_{x,p} + C_\alpha C_g  \ltn \theta,\theta^\prime \rtn_{x,p}\\
		&\leq&  2(C_f + \|A\|)(T-a)+10 C_g^2 \ltn \X\rtn_{\frac{p}{2}{\rm -var}}+ 12C_g^2 \ltn [x]\rtn_{\frac{p}{2}{\rm -var}} \notag\\
		&&+ \Big(10 C_g + 30 K_\alpha C_g^2 \ltn [x] \rtn_{\frac{p}{2}{\rm -var}} + 5 C_\alpha C_g^2 \ltn x \rtn_{p{\rm -var}}^2  + 96 C_\alpha C_g^2 \ltn \X\rtn_{\frac{p}{2}{\rm -var}} \Big) C_g \ltn x \rtn_{p{\rm -var}} \\	
		&&+ \Big(2 C_g + 6 K_\alpha C_g^2 \ltn [x] \rtn_{\frac{p}{2}{\rm -var}} + C_\alpha C_g^2 \ltn x \rtn_{p{\rm -var}}^2 + 24 C_\alpha C_g^2 \ltn \X\rtn_{\frac{p}{2}{\rm -var}}\Big) \ltn y,y^\prime\rtn_{x,p}  \\
		&&+ \Big(4 C_g + 12 K_\alpha C_g^2 \ltn [x] \rtn_{\frac{p}{2}{\rm -var}} + 2 C_\alpha C_g^2 \ltn x \rtn_{p{\rm -var}}^2 + 36 C_\alpha C_g^2 \ltn \X\rtn_{\frac{p}{2}{\rm -var}} \Big) \ltn \theta,\theta^\prime\rtn_{x,p}. 
	\end{eqnarray*}
	Using \eqref{Mest},	it is easy to check that
	\begin{eqnarray*}
		&&\ltn R^{F(\theta,\theta^\prime)}\rtn_{\frac{p}{2}{\rm -var}} + \ltn G(y,\theta) \rtn_{p{\rm -var}} \\
		&\leq& 	M \Big\{T-a +  \ltn \X \rtn_{\frac{p}{2}{\rm -var}} + \ltn [x] \rtn_{\frac{p}{2}{\rm -var}} +\ltn x \rtn_{p{\rm -var}}\\
		&&+  \Big(\ltn \X \rtn_{\frac{p}{2}{\rm -var}} + \ltn x \rtn_{p{\rm -var}}^2 +  \ltn [x] \rtn_{\frac{p}{2}{\rm -var}}\Big) \ltn x \rtn_{p{\rm -var}} \Big\} \Big(1+ \ltn \theta,\theta^\prime \rtn_{x,p} + \ltn y,y^\prime \rtn_{x,p}\Big) \\
		&\leq & 	M \Big\{T-a +  \ltn \X \rtn_{\frac{p}{2}{\rm -var}} + \ltn [x] \rtn_{\frac{p}{2}{\rm -var}} +\ltn x \rtn_{p{\rm -var}} \\
		&&+  \Big(\ltn \X \rtn_{\frac{p}{2}{\rm -var}} + \ltn x \rtn_{p{\rm -var}}^2 +  \ltn [x] \rtn_{\frac{p}{2}{\rm -var}}\Big) \ltn x \rtn_{p{\rm -var}} \Big\}\Big(1+ \ltn y,y^\prime \rtn_{x,p}\Big) \Big(1+ \ltn \theta,\theta^\prime \rtn_{x,p} \Big). 
	\end{eqnarray*}
	Now construct for any fixed $\mu \in (0,1)$ a sequence of stopping times $\{\tau_k\}_{k \in \N}$ such that $\tau_0 = a$ and
	\begin{eqnarray}\label{stoppingtime}
	&&\tau_{k+1}-\tau_k +\ltn x \rtn_{p{\rm -var},[\tau_k,\tau_{k+1}]} +  \ltn \X \rtn_{\frac{p}{2}{\rm -var},\Delta^2([\tau_k,\tau_{k+1}])}+ \ltn [x] \rtn_{\frac{p}{2}{\rm -var},\Delta^2([\tau_k,\tau_{k+1}])} \notag\\ &&=\frac{\mu}{2M(1+\ltn y,y^\prime \rtn_{x,p,[a,b]} )}<1, 
	\end{eqnarray}
	for all $k \in \N$, then it follows that
	\[
	\ltn Q(y,\theta) \rtn_{p{\rm -var},[\tau_k,\tau_{k+1}]} + \ltn R^{F(\theta,\theta^\prime)}\rtn_{\frac{p}{2}{\rm -var},[\tau_k,\tau_{k+1}]} \leq \mu  + \mu  \ltn \theta,\theta^\prime \rtn_{x,p}. 
	\]
	Hence using the fact that $\theta^\prime = Q(y,\theta)$ and $F(\theta,\theta^\prime) = \theta$ we conclude that
	\begin{equation}
	\ltn (\theta,\theta^\prime) \rtn_{x,p,[\tau_k,\tau_{k+1}]} \leq  \frac{\mu }{1-\mu}. 
	\end{equation}
	Therefore by applying Lemma \ref{additive}, it follows that
	\begin{eqnarray*}
	\ltn (\theta,\theta^\prime) \rtn_{x,p,[a,b]}&\leq& \sum_{i=0}^{N-1} \Big(N^{\frac{p-1}{p}} \ltn \theta^\prime \rtn_{p{\rm -var},[\tau_i,\tau_{i+1}]} + N^{\frac{p-2}{p}} \ltn R^\theta \rtn_{\frac{p}{2}{\rm -var},[\tau_i,\tau_{i+1}]}\Big)\\
	&\leq& \frac{\mu}{1-\mu} N^{1+\frac{p-1}{p}} \leq \frac{\mu}{1-\mu} N^2, 
	\end{eqnarray*}
	where $N= N_{\frac{\mu}{2M(1+\ltn y,y^\prime \rtn_{x,p})},[a,b],p}(\bx)$ is the number of greedy times $\tau_k$ defined in \eqref{stoppingtime} in the interval $[a,b]$. It is easy to see that
	\begin{eqnarray*}
		b-a &>& N_{\frac{\mu}{2M(1+\ltn y,y^\prime \rtn_{x,p})},[a,b],p}(\bx)\times \\
		&&\times \Big\{ \frac{\mu}{2M(1+  \ltn y,y^\prime \rtn_{x,p,[a,b]})} \Big( 1 + \ltn x \rtn_{\nu,[a,b]} +  \ltn \X \rtn_{2\nu,\Delta^2([a,b])} +  \ltn [x] \rtn_{2\nu,\Delta^2([a,b])} \Big)^{-1} \Big\}^{\frac{1}{\nu}}.
	\end{eqnarray*}
	All in all, we have just shown that for all $0\leq a \leq b \leq 1$
	\begin{eqnarray*}\label{alphanorm0}
	&&\ltn (\theta,\theta^\prime) \rtn_{x,p,[a,b]} \notag\\
	&\leq&  \frac{(b-a)^2(2M)^{\frac{2}{\nu}}}{(1-\mu)\mu^{\frac{2}{\nu}-1}}  \Big\{\Big( 1 + \ltn x \rtn_{\nu,[a,b]} +  \ltn \X \rtn_{2\nu,[a,b]} +  \ltn [x]\rtn_{2\nu,[a,b]}\Big)\Big(1+ \ltn y,y^\prime \rtn_{x,p,[a,b]}\Big)\Big\}^{\frac{2}{\nu}}\notag \\
	&\leq&  \frac{(b-a)^2(8M)^{\frac{2}{\nu}}}{4(1-\mu)\mu^{\frac{2}{\nu}-1}} \Big[1+ \ltn y,y^\prime \rtn_{x,p,[a,b]}^\frac{2}{\nu}\Big]\Big[1+\Big(\ltn x \rtn_{\nu,[a,b]} +  \ltn \X \rtn_{2\nu,\Delta^2([a,b])} +  \ltn [x]\rtn_{2\nu,\Delta^2([a,b])} \Big)^\frac{2}{\nu} \Big] \notag \\
	&\leq&  \frac{(b-a)^2(8M)^{\frac{2}{\nu}}}{4(1-\mu)\mu^{\frac{2}{\nu}-1}}\Big\{1+ \ltn y,y^\prime \rtn_{x,p,[a,b]}^\frac{2}{\nu}+\Big(\ltn x \rtn_{\nu,[a,b]} +  \ltn \X \rtn_{2\nu,\Delta^2([a,b])} +  \ltn [x]\rtn_{2\nu,\Delta^2([a,b])} \Big)^\frac{2}{\nu} \notag\\
	&&\qquad \qquad \qquad\qquad+\frac{1}{2}\ltn y,y^\prime \rtn_{x,p,[a,b]}^\frac{4}{\nu} + \frac{1}{2} \Big(\ltn x \rtn_{\nu,[a,b]} +  \ltn \X \rtn_{2\nu,\Delta^2([a,b])} +  \ltn [x]\rtn_{2\nu,\Delta^2([a,b])} \Big)^\frac{4}{\nu}\Big\},
	\end{eqnarray*}
	which proves \eqref{alphanormtheta}.
\end{proof}	

\section*{Acknowledgments}
This work was supported by the Max Planck Institute for Mathematics in the Science (MIS-Leipzig).

\end{document}